\documentclass[12pt,reqno]{amsart}

\makeatletter
\@namedef{subjclassname@1991}{\textup{2020} Mathematics Subject Classification}
\makeatother

\usepackage{color}
\usepackage{amsmath, amsthm, amssymb, amsxtra}
\usepackage{amsfonts}
\usepackage[utf8]{inputenc}
\usepackage[dvips]{epsfig}
\usepackage{graphicx}
\usepackage[english]{babel}
\usepackage{hyperref}

\usepackage{graphicx}
\usepackage{latexsym}
\usepackage{mathtools}
\usepackage{esint}
\usepackage{float}

\theoremstyle{plain}
\newtheorem{thm}{Theorem}[section]
\newtheorem{cor}[thm]{Corollary}
\newtheorem{lem}[thm]{Lemma}
\newtheorem{prop}[thm]{Proposition}

\theoremstyle{definition}
\newtheorem{defn}[thm]{Definition}

\theoremstyle{remark}
\newtheorem{obs}[thm]{Remark}

\numberwithin{equation}{section}

\newcommand{\average}{{\mathchoice {\kern1ex\vcenter{\hrule height.4pt
width 6pt depth0pt} \kern-9.7pt} {\kern1ex\vcenter{\hrule
height.4pt width 4.3pt depth0pt} \kern-7pt} {} {} }}

\newcommand{\R}{\mathbb R}
\renewcommand{\L}{\mathcal L}

\newcommand{\p}{\partial}

\newcommand{\comment}[1]{}

\begin{document}

\title[New boundary Harnack inequalities with right hand side]{New boundary Harnack inequalities\\ with right hand side}
\author{Xavier Ros-Oton}
\address{ICREA, Pg.\ Llu\'is Companys 23, 08010 Barcelona, Spain \&\newline\indent
Universitat de Barcelona, Departament de Matem\`atiques i Inform\`atica, Gran Via de les Corts Catalanes 585, 08007 Barcelona, Spain.}
\email{xros@ub.edu}
\author{Dami\`a Torres-Latorre}
\address{Universitat de Barcelona, Departament de Matem\`atiques i Inform\`atica, Gran Via de les Corts Catalanes 585, 08007 Barcelona, Spain.}
\email{\tt dtorres-latorre@ub.edu}

\begin{abstract}
We prove new boundary Harnack inequalities in Lipschitz domains for equations with a right hand side. Our main result applies to non-divergence form operators with bounded measurable coefficients and to divergence form operators with continuous coefficients, whereas the right hand side is in $L^q$ with $q > n$. Our approach is based on the scaling and comparison arguments of \cite{DS20}, and we show that all our assumptions are sharp.

As a consequence of our results, we deduce the $\mathcal{C}^{1,\alpha}$ regularity of the free boundary in the fully nonlinear obstacle problem and the fully nonlinear thin obstacle problem.
\end{abstract}

\thanks{XR and DT have received funding from the European Research Council (ERC) under the Grant Agreement No 801867. XR was supported by the Swiss National Science Foundation, and by MINECO grant MTM2017-84214-C2-1-P (Spain).}
\maketitle

\vspace{0.5cm}

\section{Introduction}\label{sect:intro}
\addtocontents{toc}{\protect\setcounter{tocdepth}{1}}
\subsection{Background}
The boundary Harnack inequality states that all positive harmonic functions with zero boundary condition are locally comparable as they approach the boundary, under appropriate assumptions on the domain. More precisely, if $u$ and $v$ are positive harmonic functions in $\Omega$ that vanish on $\p\Omega$, then
$$C^{-1} \leq \frac{u}{v} \leq C,$$
with $C$ depending on the dimension and $u(p)/v(p)$ for a fixed interior point $p$.

Notice that such a result is most relevant in domains that are less regular than $\mathcal{C}^{1,\mathrm{Dini}}$, because otherwise the Hopf lemma combined with the $\mathcal{C}^1(\overline{\Omega})$ regularity of the solutions yields the same conclusion, see for example \cite{LZ18}.

The boundary Harnack inequality is known to be true for a broad class of domains and for solutions of more general elliptic equations. The classical case for harmonic functions was first proved by Kemper in Lipschitz domains in \cite{Kem72}. Operators in divergence form were first considered by Caffarelli, Fabes, Mortola and Salsa in \cite{CFMS81} in Lipschitz domains, while the case of operators in non-divergence form was treated in \cite{FGMS88} by Fabes, Garofalo, Marin-Malave and Salsa. Jerison and Kenig extended the same result to NTA domains in the case of divergence form operators in \cite{JK82}. On the other hand, the case of non-divergence operators in Hölder domains with $\alpha > 1/2$ was treated with probabilistic techniques in \cite{BB94} by Bass and Burdzy. Recently, De Silva and Savin found a simple and unified proof of all these previous results in \cite{DS20}.

Besides, Allen and Shahgholian recently proved the boundary Harnack for divergence form equations \textbf{with right hand side} in Lipschitz domains \cite{AS19}, under appropriate assumptions on the operator, the right hand side and the domain. In particular, in the case of the Laplacian, their result implies that if the $L^\infty$ norm of the right hand side and the Lipschitz constant of the domain are small enough, then the boundary Harnack inequality still holds. This enables using the classical proof in \cite{Caf98} due to Caffarelli (see also \cite[Section 6.2]{PSU12} or \cite[Section 5.4]{FR20}) of the regularity of the free boundary in the obstacle problem $\Delta u = \chi_{\{u > 0\}}$ in the more general case $\Delta u = f\chi_{\{u > 0\}}$, with $f$ Lipschitz; see \cite[Section 1.4.2]{AS19}.

Here, we extend such boundary Harnack inequality to \textbf{non-divergence} equations with possibly \textbf{unbounded} right hand side in $L^q$, with $q > n$. (This was only known in $\mathcal{C}^{1,1}$ domains \cite{Sir17, Sir20}.) This allows us to use the classical proof of the free boundary regularity in the obstacle problem $\Delta u = f\chi_{\{u > 0\}}$ to the case $f \in W^{1,q}$, and can also be applied to fully nonlinear free boundary problems of the form
\begin{equation}\label{eq:fully_nonlinear_intro}
F(D^2u) = f\chi_{\{u > 0\}}
\quad \text{or}\quad
\left\{
\begin{array}{rcll}
F(D^2v) & = & 0 & \text{in } \{v > \varphi\}\\
F(D^2v) & \leq & 0 &\\
v & \geq & \varphi. &
\end{array}
\right.
\end{equation}

Moreover, we also establish a boundary Harnack for equations with a right hand side in \textit{slit domains}, and use it to establish the $\mathcal{C}^{1,\alpha}$ regularity of the free boundary in the fully nonlinear thin obstacle problem, a question left open in \cite{RS16}.

\subsection{Setting}
In the following, $\mathcal{L}$ will denote either a non-divergence form elliptic operator with \textbf{bounded measurable} coefficients,
\begin{equation}\label{eq:non-divergence_operator}
\mathcal{L}u = \operatorname{Tr}(A(x)D^2u), \quad \text{with} \quad \lambda I \leq A(x) \leq \Lambda I,
\end{equation}
with $0 < \lambda \leq \Lambda$, or a divergence form elliptic operator with \textbf{continuous} coefficients,
\begin{equation}\label{eq:divergence_operator}
\mathcal{L}u = \operatorname{Div}(A(x)\nabla u), \quad \text{with} \quad \lambda I \leq A(x) \leq \Lambda I \quad \text{and} \quad A \in \mathcal{C}^0,
\end{equation}
where $A$ has modulus of continuity $\sigma$, and $0 < \lambda \leq \Lambda$.

We will consider Lipschitz domains of the following form, where $B_1'$ is the unit ball of $\R^{n-1}$.

\begin{defn}\label{defn:lipschitz_domain}
We say $\Omega$ is a Lipschitz domain with  Lipschitz constant $L$ if $\Omega$ is the epigraph of a Lipschitz function $g : B_1' \to \R$, with $g(0) = 0$:
$$\Omega = \big\{(x', x_n) \in B_1'\times\R \quad \text{such that} \quad x_n > g(x')\big\}, \quad \|g\|_{\mathcal{C}^{0,1}} = L.$$
\end{defn}

\subsection{Main results}
We present here our new boundary Harnack inequality. 

We emphasize that the following result applies to both non-divergence and divergence form operators, and that the only regularity assumption on the coefficients is the continuity of $A(x)$ in case of divergence-form operators. Throughout the paper, when we say $L^n$-viscosity or weak solutions, we refer to $L^n$-viscosity solutions in the case of non-divergence form operators (\ref{eq:non-divergence_operator}), and to weak solutions in the case of divergence form operators (\ref{eq:divergence_operator}).

\begin{thm}\label{thm:BH_rhs}
Let $q > n$ and $\mathcal{L}$ as in (\ref{eq:non-divergence_operator}) or (\ref{eq:divergence_operator}). There exist small constants $c_0 > 0$ and $L_0 > 0$ such that the following holds.

Let $\Omega$ be a Lipschitz domain as in Definition \ref{defn:lipschitz_domain}, with Lipschitz constant $L < L_0$. Let $u$ and $v > 0$ be solutions of
\begin{equation*}
\left\{
\begin{array}{rcll}
\mathcal{L}u & = & f & \text{in } \Omega \cap B_1\\
u & = & 0 & \text{on } \partial\Omega \cap B_1
\end{array}
\right.
\ \text{and}\quad
\left\{
\begin{array}{rcll}
\mathcal{L}v & = & g & \text{in } \Omega \cap B_1\\
v & = & 0 & \text{on } \partial\Omega \cap B_1,
\end{array}
\right.
\end{equation*}
in the $L^n$-viscosity or the weak sense, with
\begin{equation}\label{eq:f_g_fitades}
\|f\|_{L^q(B_1)} \leq c_0,\quad \|g\|_{L^q(B_1)} \leq c_0.
\end{equation}

Additionally, assume that $v(e_n/2) \geq 1$ and either $u > 0$ and $u(e_n/2) \leq 1$, or $\|u\|_{L^p(B_1)} \leq 1$ for some $p > 0$.

Then,
$$u \leq Cv \quad \text{in} \quad B_{1/2},$$
and
$$\left\|\frac{u}{v}\right\|_{\mathcal{C}^{0,\alpha}(\overline{\Omega}\cap B_{1/2})} \leq C.$$
The constants $C$, $c_0$, $L_0$ and $\alpha > 0$ depend only on the dimension, $q$, $\lambda$, $\Lambda$, as well as $p$ and $\sigma$, when applicable.
\end{thm}

\begin{obs}
All the hypotheses of the theorem are optimal in the following sense:
\begin{itemize}
    \item If the Lipschitz constant $L_0$ of the domain is not small, the theorem fails, even for $q = \infty$ and for $\mathcal{L} = \Delta$.
    \item If $q = n$, the theorem fails for any $c_0 > 0$ and any $L_0 > 0$, even for $\mathcal{L} = \Delta$.
    \item The result fails in general for operators in divergence form with bounded measurable coefficients.
\end{itemize}
We provide counterexamples to plausible extensions in this sense in Section \ref{sect:counterexample}.
\end{obs}

When the two functions are positive, we recover the standard symmetric formulation of the boundary Harnack.

\begin{cor}\label{cor:BH_rhs_symmetric}
Let $q > n$ and $\mathcal{L}$ as in (\ref{eq:non-divergence_operator}) or (\ref{eq:divergence_operator}). There exist small constants $c_0 > 0$ and $L_0 > 0$ such that the following holds. Let $\Omega$ be a Lipschitz domain as in Definition \ref{defn:lipschitz_domain}, with Lipschitz constant $L < L_0$. Let $u, v$ be positive solutions of
\begin{equation*}
\left\{
\begin{array}{rcll}
\mathcal{L}u & = & f & \text{in } \Omega \cap B_1\\
u & = & 0 & \text{on } \partial\Omega \cap B_1
\end{array}
\right.
\ \text{and}\quad
\left\{
\begin{array}{rcll}
\mathcal{L}v & = & g & \text{in } \Omega \cap B_1\\
v & = & 0 & \text{on } \partial\Omega \cap B_1,
\end{array}
\right.
\end{equation*}
in the $L^n$-viscosity or the weak sense, with $f$ and $g$ satisfying (\ref{eq:f_g_fitades}).

Assume $u, v$ are normalized in the sense that $u(e_n/2) = v(e_n/2) = 1$. Then,
$$C^{-1} \leq \frac{u}{v} \leq C \quad \text{in} \quad B_{1/2},$$
and
$$\left\|\frac{u}{v}\right\|_{\mathcal{C}^{0,\alpha}(\overline{\Omega}\cap B_{1/2})} \leq C.$$
The positive constants $C$, $c_0$, $L_0$ and $\alpha$ depend only on the dimension, $q$, $\lambda$, $\Lambda$, as well as $\sigma$, when applicable.
\end{cor}

\subsection{Applications to obstacle problems}
The boundary Harnack inequality is the technical tool that allows us to prove $\mathcal{C}^{1,\alpha}$ regularity of the free boundary once we know it is Lipschitz in the classical obstacle problem with constant right hand side \cite[Section 5.6]{FR20} and in the thin obstacle problem with zero obstacle \cite[Section 5]{Fer20}.

The functions to which we apply the boundary Harnack are derivatives of the solution to the free boundary problem. Hence, if the original free boundary problem is the classical obstacle problem,
\begin{equation}\label{eq:classical_obstacle}
\left\{
\begin{array}{rcl}
\Delta u & = & f\chi_{\{u > 0\}}\\
u & \geq & 0,
\end{array}
\right.
\end{equation}
the derivatives of $u$ are solutions of
\begin{equation*}
\left\{
\begin{array}{rcll}
\Delta (\partial_\nu u) & = & \partial_\nu f & \text{in } \{u > 0\}\\
\partial_\nu u & = & 0 & \text{on } \p\{u > 0\},
\end{array}
\right.
\end{equation*}
and we can apply the boundary Harnack of Allen and Shahgholian if $f \in W^{1,\infty}$ (Lipschitz), or our new Theorem \ref{thm:BH_rhs} if $f \in W^{1,q}$ with $q > n$.

In the fully nonlinear setting (\ref{eq:fully_nonlinear_intro}), the derivatives of the solution satisfy a linear equation in non-divergence form,
$$\mathcal{L}(\partial_\nu u) = g \quad \text{in} \quad \{u > 0\},$$
with bounded measurable coefficients $A(x)$, and then having our new boundary Harnack for non-divergence operators proves useful to deduce results on the regularity of the free boundary.

It is well known that the free boundary may exhibit singularities. Hence, we need to introduce the notion of a regular point.

\begin{defn}\label{defn:regular_point_classical}
Let $x_0$ be a free boundary point for the classical obstacle problem, i.e.\ $x_0 \in \p\{u > 0\}$ for a solution of (\ref{eq:classical_obstacle}). We say that $x_0$ is a regular free boundary point if there exists $r_k \downarrow 0$ such that
$$\frac{u(r_kx)}{r_k^2} \rightarrow \frac{\gamma}{2}(x\cdot e)_+^2 \quad \text{in } \mathcal{C}^1_{\mathrm{loc}}(\R^n)$$
for some $\gamma > 0$ and $e \in \mathbb{S}^{n-1}$.
\end{defn}

Our next application was already known by using perturbative arguments with slightly weaker assumptions  \cite{Bla01}. We include this result to illustrate the arguments that we will use in the fully nonlinear problems in a more easily readable setting.

\begin{cor}\label{cor:classical_obstacle_lq}
Let $u$ be a solution of (\ref{eq:classical_obstacle}) with $f \geq c_0 > 0$ in $W^{1,q}(B_1)$, with $q > n$, and assume the origin is a regular free boundary point in the sense of Definition \ref{defn:regular_point_classical}.

Then, the free boundary $\Gamma = \partial\{u > 0\}$ is locally a $\mathcal{C}^{1,\alpha}$ graph at $0$.
\end{cor}

The fully nonlinear obstacle problem can be presented in at least two different formulations. The following one was studied by Lee in \cite{Lee98}.
\begin{equation}\label{eq:fully_nonlinear_thick_phi_informal}
    \left\{
    \begin{array}{rcl}
    F(D^2v) & \leq & 0\\
    v & \geq & \varphi\\
    F(D^2v) & = & 0 \quad \text{in} \quad \{v > \varphi\}.
    \end{array}
    \right.
\end{equation}
Here, we impose the following conditions:
\begin{itemize}
    \item $F$ is uniformly elliptic.
    \item $F(D^2\varphi) \leq -\tau_0 < 0$.
    \item $\varphi \in \mathcal{C}^\infty$.
\end{itemize}
Then, under these hypotheses, $v \in \mathcal{C}^{1,1}$ and the free boundary is $\mathcal{C}^{1,\alpha}$ at regular points. For our purposes, we will say a free boundary point is regular in the sense of Definition \ref{defn:regular_point_classical}, as in the classical obstacle problem.

More generally, one can study problems of the form
\begin{equation}\label{eq:fully_nonlinear}
\begin{cases}
F(D^2u, x) = f\chi_{\{u > 0\}}\\
u \geq 0.
\end{cases}
\end{equation}
This is a generalization of problem (\ref{eq:fully_nonlinear_thick_phi_informal}). Indeed, if we define $u = v - \varphi$, then
$$\tilde{F}(D^2u,x) := F(D^2u + D^2\varphi) - F(D^2\varphi) = -F(D^2\varphi) =: f(x) \quad \text{in} \quad \{u > 0\}.$$

This fully nonlinear obstacle problem (and more general ones without the sign condition on $u$) has been further studied by Lee, Shahgholian, Figalli, and more recently by Indrei and Minne in \cite{LS01,FS14, IM16}. They proved that if $F$ is convex, $f$ is Lipschitz and $f \geq \tau_0 > 0$, the free boundary $\partial\Omega$ is $\mathcal{C}^1$ at regular points.

As a consequence of our new boundary Harnack inequality, we extend their result for (\ref{eq:fully_nonlinear}) in two ways. We lower the Lipschitz regularity required for $f$ to $W^{1,q}$ with $q > n$, and we prove $\mathcal{C}^{1,\alpha}$ regularity of the free boundary instead of $\mathcal{C}^1$.

\begin{cor}\label{cor:fully_nonlinear_thick}
Let $u$ be a solution of (\ref{eq:fully_nonlinear}). Assume as well:
\begin{itemize}
    \item[\textit{(H1)}] $F$ is uniformly elliptic and $F(0,x) = 0$ for all $x \in \Omega$.
    \item[\textit{(H2)}] $F$ is convex and $\mathcal{C}^1$ in the first variable, and $W^{1,q}$ in the second variable for some $q > n$.
    \item[\textit{(H3)}] $f \in W^{1,q}$ for some $q > n$, and $f \geq \tau_0 > 0$.
\end{itemize}

Then, if the origin is a regular free boundary point in the sense of Definition \ref{defn:regular_point_classical}, the free boundary is a $\mathcal{C}^{1,\alpha}$ graph in $B_r$ for some small $r > 0$ and $\alpha > 0$.
\end{cor}

\subsection{Thin obstacle problems}
The thin obstacle problem, also known as the Signorini problem, is a classical free boundary problem that admits several formulations, see \cite{Fer20} for a nice introduction to the topic. One can write the problem as the following, given an obstacle $\varphi$ defined on $\{x_n = 0\}$:
\begin{equation}\label{eq:classical_signorini}
\left\{
\begin{array}{rcll}
\Delta v & \leq & 0 & \text{in } B_1\\
v & \geq & \varphi & \text{on } B_1 \cap \{x_n = 0\}\\
\Delta v & = & 0 & \text{in } B_1 \setminus \{(x',0) : v(x',0) = \varphi(x')\}.
\end{array}
\right.
\end{equation}

The first results on regularity of the solution $v$ were established in the seventies, in particular it was proved in \cite{Caf79} that $v \in \mathcal{C}^{1,\alpha}$ for a small $\alpha > 0$. Free boundary regularity remained open for quite some time, until the first free boundary regularity result, \cite{ACS08}, establishing that the free boundary is $\mathcal{C}^{1,\alpha}$ at \textit{regular} points when $\varphi \equiv 0$. Further results have been obtained in \cite{KPS15, DS16} among others, proving that the free boundary is real analytic at regular points provided that $\varphi$ is analytic.

Consider now the fully nonlinear thin obstacle problem.
\begin{equation}\label{eq:fully_nonlinear_signorini}
\left\{
\begin{array}{rcll}
F(D^2v) & \leq & 0 & \text{in } B_1\\
v & \geq & \varphi & \text{on } B_1 \cap \{x_n = 0\}\\
F(D^2v) & = & 0 & \text{in } B_1 \setminus \{(x',0) : v(x',0) = \varphi(x')\},
\end{array}
\right.
\end{equation}
where $F$ is uniformly elliptic, convex and $F(0) = 0$.

Milakis and Silvestre proved in \cite{MS08} that solutions $u$ are $\mathcal{C}^{1,\alpha}$ in the symmetric case (even functions with respect to $x_n$). More recently, Fernández-Real extended the result to the non-symmetric case in \cite{Fer16}. The first result on free boundary regularity is due to the first author and Serra \cite{RS16}, where they proved the $\mathcal{C}^1$ regularity of the free boundary near regular points. Here, we will prove for the first time that the free boundary is actually $\mathcal{C}^{1,\alpha}$.

To do this, we need to adapt Theorem \ref{thm:BH_rhs} to the case of slit domains. We present here a simplified version, see Section \ref{sect:slit} for a more general result.

\begin{thm}\label{thm:BH_rhs_slit_intro_hiperpla}
Let $q > n$ and let $\mathcal{L}$ be as in (\ref{eq:non-divergence_operator}). There exists small $c_0 > 0$ such that the following holds.

Let $\Omega = B_1 \setminus K$ with $K$ a closed subset of $\{x_n = 0\}$. Let 
$$\Omega^+ = \Omega \cap \{x_n \geq 0\} \quad \text{and} \quad \Omega^- = \Omega \cap \{x_n \leq 0\}.$$
Let $u$ and $v > 0$ be $L^n$-viscosity solutions of
\begin{equation*}
\left\{
\begin{array}{rcll}
\mathcal{L}u & = & f & \text{in } B_1 \setminus K\\
u & = & 0 & \text{on } K
\end{array}
\right.
\ \text{and}\quad
\left\{
\begin{array}{rcll}
\mathcal{L}v & = & g & \text{in } B_1 \setminus K\\
v & = & 0 & \text{on } K,
\end{array}
\right.
\end{equation*}
with $f$ and $g$ satisfying (\ref{eq:f_g_fitades}). Assume in addition that $v(e_n/2) \geq 1$, $v(-e_n/2) \geq 1$, and either $u > 0$ in $B_1 \setminus K$ and $\max\{u(e_n/2),u(-e_n/2)\} \leq 1$, or $\|u\|_{L^p(B_1)} \leq 1$ for some $p > 0$. Then,
$$u \leq Cv \quad \text{in} \quad B_{1/2} \setminus K,$$
and
$$\left\|\frac{u}{v}\right\|_{\mathcal{C}^{0,\alpha}(\overline{\Omega^\pm}\cap B_{1/2})} \leq C.$$
The positive constants $C$, $c_0$, and $\alpha$ depend only on the dimension, $q$, $\lambda$, $\Lambda$, as well as $p$, when applicable.
\end{thm}

Using this new boundary Harnack, we can prove the following.

\begin{cor}\label{cor:fully_nonlinear_thin}
Assume that $0$ is a regular free boundary point for (\ref{eq:fully_nonlinear_signorini}) in the sense of \cite{RS16}, with $F \in \mathcal{C}^1$ and $\varphi \in W^{3,q}$ for some $q > n$. Then, there exists $\rho > 0$ such that the free boundary is a $\mathcal{C}^{1,\alpha}$ graph in $B_\rho \cap \{x_n = 0\}$.
\end{cor}

This is new, even when $\varphi \in \mathcal{C}^\infty$. The higher regularity of the free boundary remains a challenging open question.

\subsection{Plan of the paper}
The paper is organized as follows.

In Section \ref{sect:preliminaries}, we recall some classical results and tools, such as the ABP estimate and the weak Harnack inequality. Then, in Section \ref{sect:main} we prove our new boundary Harnack inequality for elliptic equations with right hand side, Theorem \ref{thm:BH_rhs}, by scaling and barrier arguments. Section \ref{sect:slit} is devoted to adapting the result to slit domains. In Section \ref{sect:applications}, we prove the $\mathcal{C}^{1,\alpha}$ regularity of the free boundary in the fully nonlinear obstacle problem, Corollary \ref{cor:fully_nonlinear_thick}, and in the fully nonlinear thin obstacle problem, Corollary \ref{cor:fully_nonlinear_thin}. Finally, in Section \ref{sect:counterexample}, we present two counterexamples that show the sharpness of our new boundary Harnack and in Section \ref{sect:hopf} we introduce a Hopf lemma for equations with right hand side.

\section{Preliminaries}\label{sect:preliminaries}
In this section we recall some classical tools and results that will be used throughout the paper. We will denote
$$\mathcal{M^-}(D^2u) := \inf\limits_{\lambda I \leq A \leq \Lambda I}\operatorname{Tr}(AD^2u), \quad \mathcal{M^+}(D^2u) := \sup\limits_{\lambda I \leq A \leq \Lambda I}\operatorname{Tr}(AD^2u)$$
the Pucci extremal operators, see \cite{CC95} or \cite{FR20} for their properties. 

\subsection{$L^n$-viscosity and weak solutions}
In this work we are considering linear elliptic equations of the form $\mathcal{L}u = f$, with $f \in L^q$, with $q \geq n$. The most appropriate notion of solutions for a divergence form equation are the well-known weak solutions.

For the non-divergence form case, one could consider \textit{strong} ($W^{2,n}_{\mathrm{loc}}$, solving the PDE in the a.e.\ sense) solutions, but all the arguments of the proof are equally viable for $L^n$-viscosity solutions, which are more general. We present the minimal definition for the linear case.

\begin{defn}[\cite{CCKS96}]\label{defn:Ln-viscosity}
Let $u \in \mathcal{C}(\Omega)$, $f \in L^n_{\mathrm{loc}}(\Omega)$ and $\mathcal{L}$ in non-divergence form. We say $u$ is a $L^n$-viscosity subsolution (resp. supersolution) if, for all ${\varphi \in W^{2,n}_{\mathrm{loc}}(\Omega)}$ such that $u - \varphi$ has a local maximum (resp. minimum) at $x_0$,
$$\operatorname{ess}\liminf\limits_{x \rightarrow x_0} \mathcal{L}\varphi - f \leq 0$$
$$(\text{resp.} \ \operatorname{ess}\limsup\limits_{x \rightarrow x_0} \mathcal{L}\varphi - f \geq 0).$$

We will say equivalently that $u$ is a solution of $\L u \leq (\geq) f$. When $u$ is both a subsolution and a supersolution, we say $u$ is a solution and write $\L u = f$.
\end{defn}

$L^n$-viscosity solutions coincide with strong, viscosity or even classical solutions when they have the required regularity, and satisfy the maximum and comparison principles, but are more flexible, for example, allowing to compute limits under some reasonable hypotheses, and are thus preferred in some contexts.

Throughout this paper, the Dirichlet boundary conditions must be understood in the pointwise sense when we are dealing with $L^n$-viscosity solutions, and in the trace sense when we are dealing with weak solutions.

\subsection{Interior estimates}
The Alexandroff-Bakelmann-Pucci estimate is one of the main tools in regularity theory for non-divergence form elliptic equations. We refer to \cite[Theorem 3.2]{CC95} and \cite[Proposition 3.3]{CCKS96} for the full details and a proof.

\begin{thm}[ABP Estimate]\label{thm:ABP}
Assume that $\Omega \subset \R^n$ is a bounded domain. Let $\L$ be a non-divergence form operator as in (\ref{eq:non-divergence_operator}) and let $u \in \mathcal{C}(\overline{\Omega})$ satisfy $\mathcal{L}u \geq f$ in the $L^n$-viscosity sense, with $f \in L^n(\Omega)$. Assume that $u$ is bounded on $\partial\Omega$.

Then,
$$\sup\limits_\Omega u \leq \sup\limits_{\p\Omega} u + C\operatorname{diam}(\Omega)\|f\|_{L^n(\Omega)}$$
with $C$ only depending on the dimension, $\lambda$ and $\Lambda$.
\end{thm}

In the case of divergence form equations, the global boundedness of weak solutions is known in more generality. For our purposes, it is sufficient to consider the case $p = n$.

\begin{thm}[\protect{\cite[Theorem 8.16]{GT98}}]\label{thm:ABP_divergence}
Assume that $\Omega \subset \R^n$ is a bounded domain. Let $\L$ be a divergence form operator as in (\ref{eq:divergence_operator}) and let $u \in \mathcal{C}(\overline{\Omega})$ be a weak solution of $\mathcal{L}u \geq f$, with $f \in L^p(\Omega)$, $p > n/2$. Assume that $u$ is bounded on $\partial\Omega$.

Then,
$$\sup\limits_\Omega u \leq \sup\limits_{\p\Omega} u + C\|f\|_{L^p(\Omega)}$$
with $C$ only depending on the dimension, $|\Omega|$, $p$, $\lambda$ and $\Lambda$.
\end{thm}

We will need the two estimates that are classically combined to obtain the Krylov-Safonov Harnack inequality. The first one is the following weak Harnack inequality, valid for $L^n$-viscosity solutions of non-divergence form equations. We refer to \cite[Theorem 2]{Tru80} and \cite[Theorem 4.5]{KS09}.

\begin{thm}[Weak Harnack inequality]
Let $\L$ be a non-divergence form operator as in (\ref{eq:non-divergence_operator}). Let $u$ satisfy $\mathcal{L}u \leq 0$ in $\Omega$ in the $L^n$-viscosity sense and let $B_R(y) \subset \Omega$. Then, for all $\sigma < 1$,
$$\|u\|_{L^p(B_{\sigma R})} \leq C\inf\limits_{B_{\sigma R}}u,$$
where $p$ and $C$ are positive and depend only on the dimension, $\sigma$ and $\Lambda/\lambda$.
\end{thm}

Now, combining this theorem with the ABP estimate, applied to the function $1 - u$, we obtain the following result. This is also valid for divergence form equations, and sometimes known as De Giorgi oscillation lemma in that setting. The case with $f = 0$ is found in \cite[Theorem 11.2]{CS05}, and we can extend it easily to the general case using Theorem \ref{thm:ABP_divergence}.

\begin{cor}\label{cor:measure_Harnack_rhs}
Let $\L$ be as in (\ref{eq:non-divergence_operator}) or (\ref{eq:divergence_operator}). Let $r \in (0,1]$, $u \leq 1$, $\mathcal{L}u \geq f$ in $B_r$, in the $L^n$-viscosity or the weak sense, with $f \in L^n(B_r)$. Assume $|\{u \leq 0\}| \geq \eta|B_r| > 0$, and that $\|f\|_{L^n(B_r)} \leq \delta(\eta)$. Then,
$$\sup\limits_{B_{r/2}} u \leq 1 - \gamma(\eta),$$
where $\delta(\eta) > 0$ and $\gamma(\eta) \in (0,1)$ depend only on the dimension, $\lambda, \Lambda$ and $\eta$.
\end{cor}

The second estimate is the upper bound in Harnack inequality, also valid for \mbox{$L^n$-viscosity} solutions of non-divergence form equations \cite[Theorem 1]{Tru80}, \cite{KS12} and weak solutions of divergence form equations \cite{DG57, LZ17}. In the divergence form case, we can add the right hand side using Theorem \ref{thm:ABP_divergence}.

\begin{thm}[$L^\infty$ bound for subsolutions]\label{thm:mig_Harnack_sup_norma}
Let $p > 0$ and let $\L$ be as in (\ref{eq:non-divergence_operator}) or (\ref{eq:divergence_operator}). Let $\mathcal{L}u \geq f$ in $B_1$, in the $L^n$-viscosity or the weak sense. Then,
$$\sup\limits_{B_{1/2}} u \leq C_p(\|u\|_{L^p(B_1)} + \|f\|_{L^n(B_1)}),$$
where $C_p > 0$ depends only on the dimension, $p, \lambda$ and $\Lambda$.
\end{thm}

\section{Proof of Theorem \ref{thm:BH_rhs}}\label{sect:main}
\subsection{Nondegeneracy}
To study solutions of $\mathcal{L}u = f$ in a Lipschitz domain it is useful to know their behaviour in a cone. In this first part of the proof we show that, much like solutions of elliptic equations with zero Dirichlet boundary conditions separate linearly from the boundary of the domain in domains with the interior ball condition (Hopf lemma), the solutions of elliptic equations with zero Dirichlet boundary conditions separate as a power of the distance at corners, and the exponent approaches $1$ as the corners become wider.

\begin{lem}\label{lem:BH_subsolution}
Let $\mathcal{L}$ be in non-divergence form as in (\ref{eq:non-divergence_operator}). Let $\beta > 1$. There exist sufficiently small $c(\beta) > 0, \eta > 0$, only depending on the dimension, $\beta$, $\lambda$ and $\Lambda$, such that the following holds. 

Let $u$ be any solution of
\begin{equation*}
\left\{
\begin{array}{rcll}
\mathcal{L}u & \leq & c(\beta) & \text{in } C_\eta\\
u & \geq & 1 & \text{on } \{x_n = 1\} \cap \overline{C_\eta}\\
u & \geq & 0 & \text{in } \partial C_\eta,
\end{array}
\right.
\end{equation*}
where $C_\eta$ is the cone defined as
$$C_\eta := \left\{x \in \R^n : \eta|x'| < x_n < 1\right\}.$$
Then,
$$u(te_n) \geq t^\beta, \quad \forall t \in (0,1).$$
\end{lem}

\begin{proof}
Assume without loss of generality that $\beta \in (1, 2)$, because if the inequality holds for $\beta > 1$, it holds also for all $\beta' > \beta$. We will use the comparison principle with a subsolution that has the desired behaviour. Let $\varepsilon \in (0,1/20)$ to be chosen later. Notice that $\sqrt{1+\varepsilon} - \sqrt{\varepsilon} > 4/5$. Define the subsolution $\varphi$ as:

$$\varphi(x) = x_n^\beta f_\varepsilon\left(\frac{\eta|x'|}{x_n}\right), \quad f_\varepsilon(t) = \frac{\sqrt{1+\varepsilon} - \sqrt{t^2 + \varepsilon}}{\sqrt{1+\varepsilon} - \sqrt{\varepsilon}}.$$

We can readily check that $\varphi(x) = 0$ for $x \in \partial C_\eta$. It is also clear that $\varphi(x) \leq 1$ in $\{x_n = 1\} \cap C_\eta$, and that $\varphi > 0$ in $C_\eta$. Now, we need some estimates on $f_\varepsilon$ and its derivatives. For $t \in [0, 1)$,

\begin{align*}
&f_\varepsilon(t) = \frac{\sqrt{1+\varepsilon} - \sqrt{t^2 + \varepsilon}}{\sqrt{1+\varepsilon} - \sqrt{\varepsilon}} \geq \frac{\sqrt{1+\varepsilon} - t - \sqrt{\varepsilon}}{\sqrt{1+\varepsilon} - \sqrt{\varepsilon}} = 1 - \frac{t}{\sqrt{1+\varepsilon} - \sqrt{\varepsilon}} > 1 - \frac{5}{4}t\\
&f_\varepsilon'(t) = -\frac{t}{\sqrt{t^2+\varepsilon}(\sqrt{1+\varepsilon} - \sqrt{\varepsilon})} \geq \frac{-1}{\sqrt{1+\varepsilon} - \sqrt{\varepsilon}} > -\frac{5}{4}\\
&|t^{-1}f_\varepsilon'(t)| \leq \frac{1}{\sqrt{\varepsilon}(\sqrt{1+\varepsilon} - \sqrt{\varepsilon})} < \frac{5}{4}\varepsilon^{-1/2}\\
&f_\varepsilon'(t) \leq \frac{-t}{(t + \sqrt{\varepsilon})(\sqrt{1+\varepsilon} - \sqrt{\varepsilon})} \leq \frac{-t}{1 + \varepsilon} < -\frac{20}{21}t\\
&|f_\varepsilon''(t)| = \left|\frac{-\varepsilon}{(t^2 + \varepsilon)^{3/2}(\sqrt{1+\varepsilon} - \sqrt{\varepsilon})}\right| \leq \frac{1}{\sqrt{\varepsilon}(\sqrt{1+\varepsilon} - \sqrt{\varepsilon})} < \frac{5}{4}\varepsilon^{-1/2}
\end{align*}
\begin{equation*}
\begin{split}
|t^2f_\varepsilon''(t)| &= \left|\frac{-\varepsilon t^2}{(t^2 + \varepsilon)^{3/2}(\sqrt{1+\varepsilon} - \sqrt{\varepsilon})}\right| \leq \left(\frac{\varepsilon^{2/3}t^{4/3}}{t^2 + \varepsilon}\right)^{3/2}\frac{1}{\sqrt{1+\varepsilon} - \sqrt{\varepsilon}}\\
&< \left(\frac{2^{2/3}\varepsilon^{1/3}}{3}\right)^{3/2}\frac{5}{4} < \frac{1}{2}\varepsilon^{1/2}.
\end{split}
\end{equation*}

In the last inequality we used that
$$\varepsilon^{2/3}t^{4/3} = 2^{2/3}\varepsilon^{1/3}\sqrt[3]{\varepsilon(t^2/2)(t^2/2)} \leq 2^{2/3}\varepsilon^{1/3}\frac{\varepsilon + t^2/2 + t^2/2}{3} = \frac{2^{2/3}\varepsilon^{1/3}}{3}(t^2 + \varepsilon).$$

Then, we will make $\varepsilon$ small and then $\eta$ small in such a way that $\mathcal{L}\varphi \geq c(\beta)$. To make the computations easier, we will use the Pucci operator $\mathcal{M}^-$, and we will denote $t = \eta|x'|/x_n$. On the one hand, we can check that
\begin{equation*}
\begin{split}
\frac{\partial^2\varphi}{\partial x_n^2} &= x_n^{\beta - 2}((\beta^2-\beta)f_\varepsilon(t) + (2-2\beta)tf_\varepsilon'(t) + t^2f_\varepsilon''(t))\\
&> x_n^{\beta-2}\left((\beta^2-\beta)\left(1 - \frac{5}{4}t\right) + (\beta - 1)\frac{40}{21}t^2 - \frac{1}{2}\varepsilon^{1/2}\right)\\
&> x_n^{\beta-2}\left((\beta - 1)\left(\beta - \frac{5\beta}{4}t + \frac{40}{21}t^2\right) - \frac{1}{2}\varepsilon^{1/2}\right).
\end{split}
\end{equation*}

Now, we compute the discriminant of the second order polynomial that we found:
$$\operatorname{Discriminant}\left(\beta - \frac{5\beta}{4}t + \frac{40}{21}t^2\right) = \frac{25\beta^2}{16} - \frac{160\beta}{21} = \beta\left(\frac{25\beta}{16} - \frac{160}{21}\right) < 0.$$

Hence, the second order polyonmial is always positive and attains a minimum $m_\beta > 0$. Choose $\varepsilon$ such that $\varepsilon^{1/2} < (\beta-1)m_\beta$. Then,

$$\frac{\partial^2\varphi}{\partial x_n^2} > x_n^{\beta-2}\left((\beta-1)m_\beta - \frac{1}{2}\varepsilon^{1/2}\right) > x_n^{\beta-2}\frac{(\beta-1)m_\beta}{2} =: c_\beta x_n^{\beta-2}> 0$$

Consider now $i = 1, \ldots, n - 1$.
\begin{equation*}
\begin{split}
\left|\frac{\partial^2\varphi}{\partial x_i^2}\right| &= x_n^{\beta-2}\left|\eta^2 t^{-1}f_\varepsilon'(t)\frac{|x'|^2 - x_i^2}{|x'|^2} + \eta^2f_\varepsilon''(t)\frac{x_i^2}{|x'|^2}\right|\\
&\leq x_n^{\beta-2}(\eta^2|t^{-1}f_\varepsilon'(t)| + \eta^2|f_\varepsilon''(t)|)\\
&< x_n^{\beta-2}\eta^2\left(\frac{5}{4}\varepsilon^{-1/2} + \frac{5}{4}\varepsilon^{-1/2}\right) < x_n^{\beta-2}\eta^2\frac{5\varepsilon^{-1/2}}{2}.
\end{split}
\end{equation*}

Now we need to compute the crossed derivatives. We begin with
\begin{equation*}
\begin{split}
\left|\frac{\partial^2\varphi}{\partial x_i\partial x_n}\right| &= x_n^{\beta-2}\left|\eta(\beta-1)\frac{x_i}{|x'|}f_\varepsilon'(t) - \eta^2\frac{x_i}{|x'|}f_\varepsilon''(t)\right|\\
&\leq x_n^{\beta-2}(\eta(\beta-1)|f_\varepsilon'(t)| + \eta^2|f_\varepsilon''(t)|)\\
&< x_n^{\beta-2}\left(\eta\frac{5(\beta - 1)}{4} + \eta^2\frac{5}{4}\varepsilon^{-1/2}\right) < x_n^{\beta-2}(\eta+\eta^2)\frac{5\varepsilon^{-1/2}}{2}.
\end{split}
\end{equation*}

And finally, taking $i \neq j$ in $\{1, \ldots, n-1\}$,
\begin{equation*}
\begin{split}
\left|\frac{\partial^2\varphi}{\partial x_i\partial x_j}\right| &= x_n^{\beta-2}\left|-\eta^2 t^{-1}f_\varepsilon'(t)\frac{x_ix_j}{|x'|^2} + \eta^2f_\varepsilon''(t)\frac{x_ix_j}{|x'|^2}\right|\\
&\leq x_n^{\beta-2}(\eta^2|t^{-1}f_\varepsilon'(t)| + \eta^2|f_\varepsilon''(t)|)\\
&< x_n^{\beta-2}\eta^2\left(\frac{5}{4}\varepsilon^{-1/2} + \frac{5}{4}\varepsilon^{-1/2}\right) < x_n^{\beta-2}\eta^2\frac{5\varepsilon^{-1/2}}{2}.
\end{split}
\end{equation*}

Define $H(x) = D^2\varphi(x)$, and also $H_0(x)$ to be the matrix with $\partial^2\varphi/\partial x_n^2$ at the lower right corner and zeros in all other entries. On the one hand, by the definition of $\mathcal{M}^-$:
$$\mathcal{M}^-(H_0) \geq \lambda x_n^{\beta - 2}c_\beta.$$

Moreover, using that $\|H-H_0\|$ is bounded by the sum of the coefficients,
$$\mathcal{M}^-(H) \geq \mathcal{M}^-(H_0) - \Lambda\sum_{i,j=1}^n|(H-H_0)_{ij}| \geq x_n^{\beta - 2}F(\eta),$$
where
$$F(\eta) = \lambda c_\beta - 5\Lambda(n-1)(\eta+\eta^2)\varepsilon^{-1/2} - \frac{5\Lambda(n-1)^2}{2}\eta^2\varepsilon^{-1/2}.$$

Since $\varepsilon > 0$ is fixed, we choose $\eta$ small enough such that $F(\eta) \geq \lambda c_\beta/2$. To end the proof,
$$\mathcal{M}^-(D^2\varphi) = \mathcal{M}^-(H) \geq x_n^{\beta - 2}\frac{\lambda c_\beta}{2} \geq \frac{\lambda c_\beta}{2} =: c(\beta),$$
where we use that $x_n \leq 1$ and $\beta - 2 < 0$.

By the comparison principle, we conclude that $u(te_n) \geq \varphi(te_n) = t^\beta$.
\end{proof}

\begin{obs}
The constant $L_0$ in Theorem \ref{thm:BH_rhs} is limited, in fact, by the value of $\eta$ from this lemma, because the domain must contain wide enough cones, so the Lipschitz constant of the boundary must be small enough.
\end{obs}

To prove the nondegeneracy property for solutions of divergence form equations, we proceed by approximation. The continuity assumption on the coefficients in (\ref{eq:divergence_operator}) is necessary, see Proposition \ref{prop:counterexample_div_form}.

The following lemma is a natural approximation property of divergence form equations.

\begin{lem}\label{lem:approximation}
Let $\Omega$ be a bounded Lipschitz domain and $K \subset \Omega$ a compact subset. Let $\mathcal{L}_1, \mathcal{L}_2$ be divergence form operators, and let $u_1, u_2 \in H^1(\Omega)$ be the solutions of the following Dirichlet problems
\[
\left\{
\begin{array}{rcll}
\mathcal{L}_1u_1 & = & 0 & \text{in } \Omega\\
u_1 & = & g & \text{on } \p\Omega,\\
\end{array}
\right.
and\quad
\left\{
\begin{array}{rcll}
\mathcal{L}_2u_2 & = & 0 & \text{in } \Omega\\
u_2 & = & g & \text{on } \p\Omega,\\
\end{array}
\right.
\]
with $g \in H^1(\Omega)$ and
$$\mathcal{L}_1u_1 = \operatorname{Div}(A_1(x)\nabla u_1), \quad \mathcal{L}_2u_2 = \operatorname{Div}(A_2(x)\nabla u_2).$$

Then,
$$\|u_1 - u_2\|_{L^\infty(K)} \leq C\{\|A_1 - A_2\|_{L^\infty(\Omega)}, \|A_1 - A_2\|_{L^\infty(\Omega)}^\tau\},$$
where $C > 0$ and $\tau \in (0,1)$ depend only on $K$, $\Omega$, $g$ and the ellipticity constants.
\end{lem}

\begin{proof}
Since $u_1 = u_2$ on $\p\Omega$, we can use $v = u_1 - u_2$ as a test function in $H^1_0(\Omega)$, to obtain
$$\int_\Omega \nabla u_1^\top A_1 \nabla v = \int_\Omega \nabla u_2^\top A_2 \nabla v = 0,$$
so
$$0 = \int_\Omega (\nabla u_1^\top A_1 - \nabla u_2^\top A_2) \nabla v = \int_\Omega \nabla v^\top A_1\nabla v + \nabla u_2^\top(A_1 - A_2)\nabla v$$
and thus
\begin{align*}
\lambda\|\nabla v\|_{L^2(\Omega)}^2 &\leq \int_\Omega \nabla v^\top A_1\nabla v = -\int_\Omega\nabla u_2^\top(A_1-A_2)\nabla v\\
&\leq \|A_1 - A_2\|_{L^\infty(\Omega)}\|\nabla u_2\|_{L^2(\Omega)}\|\nabla v\|_{L^2(\Omega)}.
\end{align*}

Hence, using that the $H^1$ norm of $u_2$ can be bounded by a constant depending on the domain, the ellipticity constants and the boundary data,
$$\|\nabla v\|_{L^2(\Omega)} \leq C_1\|A_1 - A_2\|_{L^\infty(\Omega)}.$$
This, combined with the Poincaré inequality, yields $\|v\|_{L^2(\Omega)} \leq C_2\|A_1 - A_2\|_{L^\infty(\Omega)}$.

On the other hand, let $\delta = d(K,\p\Omega)$ and define the enlarged compact set\linebreak $K' = \{x \in \Omega : d(x,K) \leq \delta/2\}$. By the De Giorgi-Nash-Moser theorem, we have\linebreak $\|u_i\|_{\mathcal{C}^{0,\alpha}(K')} \leq C_3$, where $\alpha$ and $C_3$ depend only on the domain, the dimension and the ellipticity constants, thus $\|v\|_{\mathcal{C}^{0,\alpha}(K')} \leq 2C_3$.

Let $p \in K$ such that $|v|$ reaches its maximum, and assume without loss of generality that $v(p) > 0$. Then, for all $x \in B_{\delta/2}$, $v(p+x) \geq v(p) - 2C_3|x|^\alpha$, and we can estimate $\|v\|_{L^2(\Omega)}$. The first observation is that $v(p+x) \geq v(p)/2$ when $x$ is small enough, quantitatively,
$$v(p+x) \geq v(p)/2 \quad \Longleftrightarrow \quad 2C_3|x|^\alpha \leq v(p)/2 \quad \Longleftrightarrow |x| \leq C_4v(p)^{1/\alpha},$$
so now we can use $v(p+x) \geq v(p)\chi_E/2$, $E = B_{C_4v(p)^{1/\alpha}}$ to obtain
\begin{align*}
\|v\|_{L^2(\Omega)} &\geq \left(\int_{B_{\delta/2}(p)}v^2\right)^{1/2} = \left(\int_{B_{\delta/2}}v(p+x)^2\right)^{1/2}\\
&\geq \left(\int_{B_{\delta/2}}v(p)^2\chi_E/4\right)^{1/2} \geq \min\{|B_{\delta/2}|,|E|\}^{1/2}v(p)/2.
\end{align*}

This presents us with two cases. When $B_{\delta/2} \subset E$, $v(p) \leq C_5\|v\|_{L^2(\Omega)}$. On the other hand, if $E \subset B_{\delta/2}$, $v(p)^{1+1/\alpha} \leq C_6\|v\|_{L^2(\Omega)}$. In either case,
$$v(p) \leq C_7\max\{\|v\|_{L^2(\Omega)},\|v\|_{L^2(\Omega)}^{\frac{\alpha}{\alpha+1}}\} \leq C\max\{\|A_1 - A_2\|_{L^\infty(\Omega)}, \|A_1 - A_2\|_{L^\infty(\Omega)}^{\frac{\alpha}{\alpha+1}}\},$$
and the result follows.
\end{proof}

As a consequence, we can prove the analogue of Lemma \ref{lem:BH_subsolution} for divergence form equations.

\begin{lem}\label{lem:div_subsolution}
Let $\mathcal{L}$ be in divergence form with continuous coefficients, with modulus of continuity $\sigma$ as in (\ref{eq:divergence_operator}). Let $\beta' > 1$. There exists sufficiently small $\eta' > 0$ such that the following holds.

Let $u$ be a solution of
\begin{equation*}
\left\{
\begin{array}{rcll}
\mathcal{L}u & \leq & 0 & \text{in } C_{2,\eta'}\\
u & \geq & 1 & \text{in } \{x_n > 1\} \cap C_{2,\eta'}\\
u & \geq & 0 & \text{on } \partial C_{2,\eta'}.
\end{array}
\right.
\end{equation*}
Then,
$$u(te_n) \geq t^{\beta'}, \quad \forall t \in (0,t_\sigma),$$
where
$$C_{2,\eta'} := \{x \in \R^n : \eta'|x'| < x_n < 2\}.$$
The constants $t_\sigma$ and $\eta'$ are positive and depend only on the dimension, $\beta'$, $\sigma$, $\lambda$ and~$\Lambda$.
\end{lem}

\begin{proof}
We will assume without loss of generality that $\beta' \in (1,2)$ and that $\mathcal{L}u = 0$ in $C_{2,\eta'}$. Let $\beta, \gamma$ such that $\beta' > \gamma > \beta > 1$. Let $\eta > 0$ be the one provided by Lemma \ref{lem:BH_subsolution} with exponent $\beta$. Let $\eta' < \eta/8$ and $k_0 \in \mathbb{Z}^+$, to be chosen later. We will prove by induction that $u(2^{-k}e_n) \geq c2^{-k\gamma}$ for all integer $k \geq k_0$ and some $c > 0$. Notice that this implies that $u(te_n) \geq c't^{\gamma}$ for some smaller $c' > 0$ by a direct application of interior Harnack. To end the proof, notice that if $t$ is small enough, since $\beta' > \gamma$,
$$u(te_n) \geq c' t^{\gamma} \geq t^{\beta'}.$$

We proceed now with the induction. First, we define the following auxiliary functions.
$$b(x) = \frac{x_n}{2^{-k}}\tilde{b}\left(\frac{|x'|}{x_n}\right), \quad \tilde{b}(t) = \begin{cases}
1 & t < B,\\
4 - 3t/B & B \leq t < 4B/3,\\
0 & \text{otherwise},
\end{cases}$$
with $B = 3/(2\eta)$. We also write $b_1(x',x_n) = \tilde{b}(|x'|/x_n)$ for convenience of the notation.

We claim that there exists $c > 0$ such that, for all integer $k \geq k_0$, $u \geq c 2^{-k\gamma}b_1$ in the $(n-1)$-dimensional ball $B'_{2^{2-k}/\eta}\times\{2^{-k}\}$.

For the first $k_0$, first observe that $u \geq 0$ everywhere by the maximum principle. Then, apply the interior Harnack inequality to the cylinder $B'_{2^{2-k}/\eta}\times[2^{-k},3/2]$, which is compactly contained in $C_{2,\eta'}$. Since $\sup u = 1$ in the cylinder, we have $u \geq m > 0$, and using that $b_1 \leq 1$, $u \geq mb_1$ in $B'_{2^{2-k}/\eta}\times\{2^{-k}\}$ and we can choose $c$ accordingly.

Now, for the inductive step, let $K = B'_{2^{1-k}/\eta}\times\{2^{-k-1}\}$, which is compactly contained in $C_{2^{-k},\eta'}$, and let $v$ and $v_0$ the solutions of the following Dirichlet problems
\[
\left\{
\begin{array}{rcll}
\mathcal{L}v & = & 0 & \text{in } C_{2^{-k},\eta'}\\
v & = & 2^{-k\gamma}c b(x) & \text{on } \p C_{2^{-k},\eta'},\\
\end{array}
\right.
\ and\quad
\left\{
\begin{array}{rcll}
\mathcal{L}_0v_0 & = & 0 & \text{in } C_{2^{-k},\eta'}\\
v_0 & = & 2^{-k\gamma}c b(x) & \text{on } \p C_{2^{-k},\eta'},\\
\end{array}
\right.
\]
with $\mathcal{L}_0v_0 := \operatorname{Div}(A_0\nabla v_0)$, $A_0 = A(0)$.

Observe that $v = v_0 = 0$ on the \textit{lateral} boundary of the cone $C_{2^{-k},\eta'}$. Then, it is clear that $u \geq v$ from the boundary conditions. Furthermore, by a rescaling of Lemma \ref{lem:approximation}, 
$$\|v - v_0\|_{L^\infty(K)} \leq 2^{-k\gamma}cC\max\{\|A - A_0\|_{L^\infty(C_{2^{-k},\eta'})}, \|A - A_0\|_{L^\infty(C_{2^{-k},\eta'})}^\tau\}.$$

For each $p \in K$, consider the cone $\mathcal{C}'$ with vertex in $(p', \eta'|p'|) \in \p C_{2,\eta'}$ and slope~$\eta$,
$$\mathcal{C}' := \{(x', x_n) \in \mathbb{R}^n : \eta|x'-p'| + \eta'|p'| < x_n < 2^{-k}\}.$$ 

Since $\eta' > \eta$, $\mathcal{C}' \subset C_{2,\eta'}$. Hence, $u \geq 0$ in $\p \mathcal{C}'$. Moreover, by construction, the top part, $\{x_n = 2^{-k}\} \cap \overline{\mathcal{C}'}$ is contained in $B'_{2^{2-k}/\eta}\times\{2^{-k}\}$. Hence, we can apply a rescaled Lemma \ref{lem:BH_subsolution} to the normalized $2^{k\gamma}v_0$, because $\mathcal{L}_0$ has constant coefficients and is also a non-divergence form operator to obtain
$$2^{k\gamma}v_0(p) \geq c\left(\frac{2^{-k-1} - \eta'|p'|}{2^{-k} - \eta'|p'|}\right)^\beta \geq c\left(\frac{2^{-k-1} - 2^{2-k}\eta'/\eta}{2^{-k} - 2^{2-k}\eta'/\eta}\right)^\beta = c\left(\frac{\eta - 8\eta'}{2\eta - 8\eta'}\right)^\beta.$$

Then, passing the information on $v_0$ to $v$,
\begin{align*}
v(p) &\geq c\left(\frac{\eta - 8\eta'}{2\eta - 8\eta'}\right)^\beta - \|v - v_0\|_{L^\infty(K)}\\
&\geq 2^{-k\gamma}c\left(\left(\frac{\eta - 8\eta'}{2\eta - 8\eta'}\right)^\beta - C\max\{\|A - A_0\|_{L^\infty(C_{2^{-k},\eta'})}, \|A - A_0\|_{L^\infty(C_{2^{-k},\eta'})}^\tau\}\right)\\
&\geq 2^{-k\gamma}c(1/2)^{\gamma},
\end{align*}
where for the last inequality we first choose a small $\eta'$ such that
$$\left(\frac{\eta-8\eta'}{2\eta-8\eta'}\right)^\beta \geq \left(\frac{1}{2}\right)^\gamma + \delta,$$
with $\delta > 0$, and then take $k$ large and use that $A$ is continuous. Hence, if $\eta'$ is small enough and $k_0$ is large enough in the first place, the inductive step holds.
\end{proof}

Now we are ready to prove Theorem \ref{thm:BH_rhs}. We divide the proof into three parts: an upper bound, a lower bound, and the proof of the $\mathcal{C}^{0,\alpha}$ regularity of the quotient.

\subsection{Upper bound}
We follow the arguments of \cite{DS20}; see also \cite{Saf10}. The first lemma is a geometric fact that will make subsequent computations easier.

\begin{lem}\label{lem:A_is_convex}
Let $\Omega$ be a Lipschitz domain as in Definition \ref{defn:lipschitz_domain}, with Lipschitz constant $L < 1/16$. Let
$$A = \{x \in \Omega \cap \overline{B_{1 - \delta}} : d(x, \partial\Omega) \geq \delta\} = \{x \in \Omega\cap B_1 : d(x, \partial(\Omega\cap B_1)) \geq \delta\},$$ 
where $\delta \in (0, 1/3)$. Then $A$ is star-shaped with respect to the point $e_n/2$.
\end{lem}

\begin{proof}
It is easy to check that $d(A, \partial(\Omega \cap B_1)) = \delta$, and that $e_n/2 \in A$ (since $L < 1/16$, $d(e_n/2, \partial\Omega) \geq 1/(2\sqrt{L^2+1}) > 1/3$).

We distinguish the \textit{upper} and the \textit{lower} boundaries of $A$ as:
$$\partial_u A = \{x \in \partial A : d(x, \partial\Omega) > \delta\}, \quad \partial_l A = \{x \in \partial A : d(x, \partial\Omega) = \delta\}.$$

The first step is proving that $\partial_l A$ is a Lipschitz graph with the same or lower Lipschitz constant. For this, consider the set $\Omega_\delta = \{x \in B'_{1-\delta}\times\R : d(x, \partial\Omega) > \delta\}$, which contains the points above $\partial_l A$. For every vertical line $l$ passing through $(x',0)$, with $x' \in B_{1-\delta}'$, the set $l \cap \Omega_\delta$ is not empty, so we can define $h : B_{1-\delta}' \to \R$ as 
$$h(x') = \inf \{x_n : d((x',x_n),\partial\Omega) > \delta\}.$$

Then, for a given $x' \in B_{1-\delta}'$, $(x',y) \in \Omega_\delta$ for all $y > h(x')$. Indeed, for every point $z = (z', z_n)$ in $\partial\Omega$, either $|z' - x'| > \delta$, and hence $d((x', y),z) > \delta$, or $|z' - x'| \leq \delta$. In this case, $z_n = g(z') \leq g(x') + L|z'-x'| < g(x') + \delta/16 \leq h(x') - \delta + \delta/16 < h(x')$, and then, $d((x',y),z) > d((x',x_n),z) = \delta$, because $y > h(x') > z_n$.

In any case, we have proven that $\Omega_\delta = \{(x',x_n) \in B_{1-\delta}'\times\R : x_n > h(x')\}$. Moreover, this shows that $\partial_l A$ is a subset of the graph of $h$. Now we want to see that $h$ is Lipschitz. Notice that we can also define $h$ with the complement set,
$$h(x') = \sup \{x_n : d((x',x_n),\partial\Omega) \leq \delta\} = \sup\{x_n : d((x',x_n),\partial\Omega) = \delta\}.$$

This can be seen as the superior envolvent of a union of spheres of radius $\delta$ centered at every point of $\partial\Omega$, hence
$$h(x') = \sup \{g(x'+t)+\sqrt{\delta^2-|t|^2}, t \in B'_\delta\}.$$

Since this is a supremum of equi-Lipschitz functions, $h$ is also Lipschitz with the same or lower constant, $L' \leq L < 1/16$. From $g(0) = 0$ we can also derive $h(0) \geq \delta$, and $h(0) \leq \delta\sqrt{L^2+1} < 1.02\delta$.

Now we will see that $A$ is star-shaped with \textit{center} at $e_n/2$, constructing a segment from $e_n/2$ to every point in $A$ that lies entirely inside $A$. Let $p \neq e_n/2$ be a point in $A$, and let $q = (q', q_n)$ be the intersection of the line through $p$ and $e_n/2$ and $\partial A$, that lies on the side of $p$ and is furthest from $e_n/2$. We will see later that there is only one intersection at each side, but considering the furthest is enough for now.

If $q$ lies in $\partial_l A$, $q_n = h(q')$. If $q$ lies in $\partial_u A$, the point is above $\partial_l A$ and $q_n > h(q')$. In any case, we have always $q_n \geq h(q')$. It is clear that the segment $\overline{(e_n/2)q}$, that can be parametrised by $\{(tq', (1-t)/2 + tq_n), t \in [0, 1]\}$ is contained in $\overline{B_{1-\delta}}$. We will prove that it lies entirely \textit{above} $\partial_l A$ (except maybe in the point $q$), so it has not other intersections with $\partial A$ besides $q$. We distinguish two cases:

If $q_n \geq 7/16$, for any point $tq'$ inside the segment joining $0$ and $q'$ in $B_{1-\delta}'$ (this means $t \in (0,1)$), using that $h$ is Lipschitz,
$$h(tq') \leq h(0) + L|tq'| < 1.02\delta + t/16 < 0.34 + t/16.$$
Moreover, the height of the segment $\overline{(e_n/2)q}$ above the point $tq'$ is
$$(1-t)/2+tq_n \geq 0.5 + (q_n - 0.5)t \geq 0.5 + (h(q') - 0.5)t \geq 0.5 - t/16,$$
and $0.5 - t/16 > 0.34 + t/16$ because $t/8 < 1/8 < 0.5-0.34 = 0.16$. Combining the two inequalities, $h(tq') < (1-t)/2 + tq_n$ as required.

On the other hand, if $q_n < 7/16$, $h(q') < 7/16$ as well. Since $h$ is Lipschitz,
$$h(tq') \leq h(q') + L|q'-tq'| < q_n + (1-t)/16 = (q_n+1/16) - (1/16)t.$$
The height of the segment $\overline{(e_n/2)q}$ above the point $tq'$ is 
$$(1-t)/2 + tq_n = 1/2 - (1/2-q_n)t,$$
and $1/2 - (1/2-q_n)t > (q_n+1/16) - (1/16)t$ for $t \in (0,1)$ by a simple calculation.

Hence, in any case the segment joining $e_n/2$ and $q$ crosses $\partial A$ at $q$ for the first time, implying $A$ is star-shaped.
\end{proof}

Now, we derive an interior Harnack inequality for domains with the shape we want to consider.

\begin{lem}\label{lem:Harnack_interior_Lipschitz_rhs}
Let $\Omega$ be a Lipschitz domain as in Definition \ref{defn:lipschitz_domain}, with Lipschitz constant $L < 1/16$. Let $\delta \in (0, 1/3)$. Let $\L$ be as in (\ref{eq:non-divergence_operator}) or (\ref{eq:divergence_operator}). Let $u$ be a positive solution, in the $L^n$-viscosity or the weak sense, of
\begin{equation*}
\left\{
\begin{array}{rcll}
\mathcal{L}u & = & f & \text{in } \Omega \cap B_1\\
u & = & 0 & \text{on } \partial\Omega \cap B_1,
\end{array}
\right.
\end{equation*}
with $f \in L^n(B_1)$. Let $A = \{x \in \Omega \cap \overline{B_{1 - \delta}} : d(x, \partial\Omega) \geq \delta\}$. Then,
$$\sup\limits_{A} u \leq C(\inf\limits_A u + \|f\|_{L^n(B_1)}),$$
with $C$ depending on the dimension, $\delta$, $\lambda$ and $\Lambda$, but not on the particular shape of~$\Omega$.
\end{lem}

\begin{proof}
Let $x \in A$, and we will denote $y = e_n/2$ to simplify the notation. Since $A \subset B_{1-\delta}$, $|x - y| < 2$. We define
$$m := \left\lceil\frac{8}{\delta}\right\rceil.$$

Take $x_0 = x, \ldots, x_m = y$ a uniform partition on the segment $\overline{xy}$. It is clear that $|x_{i+1} - x_i| < \delta / 4$. Then, consider the balls $B_\delta(x_i)$. We apply the interior Harnack inequality to obtain that
$$\sup\limits_{B_{\delta/2}(x_i)} u \leq C\left(\inf\limits_{B_{\delta/2}(x_i)}u + \delta\|f\|_{L^n(B_\delta(x_i))}\right).$$

In particular, $u(x_i) \leq C(u(x_{i+1}) + \delta\|f\|_{L^n(B_i)}) \leq C(u(x_{i+1}) + \delta\|f\|_{L^n(B_1)})$, and iterating this, $u(y) \leq C^{m + 1}u(x) + C'\|f\|_{L^n(B_1)}$. Taking the points in reverse order yields $u(x) \leq C^{m + 1}u(y) + C'\|f\|_{L^n(B_1)}$.

Now take $x, z \in A$, and apply the inequalities between $u(x)$ and $u(y)$ to $u(y)$ and $u(z)$. We can put them together finally to get
$$u(x) \leq C^{2(m+1)}u(z) + C''\|f\|_{L^n(B_1)}, \quad u(z) \leq C^{2(m+1)}u(x) + C''\|f\|_{L^n(B_1)}.$$

Finally, notice that $C$, $m$ and $C''$ do not depend on the shape of $\Omega$.
\end{proof}

The next step is the following lemma, that shows that the condition $u > 0$ and $u(e_n/2) \leq 1$ implies $\|u\|_{L^p(B_1)} \leq c_p$ in Theorem \ref{thm:BH_rhs}.

\begin{lem}\label{lem:BH_rhs_positiva_implica_lp}
Let $\Omega$ a Lipschitz domain as in Definition \ref{defn:lipschitz_domain}, with Lipschitz constant $L < 1/16$. Let $\L$ be as in (\ref{eq:non-divergence_operator}) or (\ref{eq:divergence_operator}). Let $u$ be a positive solution, in the \mbox{$L^n$-viscosity} or the weak sense, of
\begin{equation*}
\left\{
\begin{array}{rcll}
\mathcal{L}u & = & f & \text{in } \Omega \cap B_1\\
u & = & 0 & \text{on } \partial\Omega \cap B_1,
\end{array}
\right.
\end{equation*}
such that $u(e_n/2) \leq 1$, with $f \in L^n(B_1)$. Then, there exist $p, C_p > 0$ such that 
$$\|u\|_{L^p(B_1)} \leq C_p,$$
with $p$ and $C_p$ only depending on the dimension, $\lambda$, $\Lambda$ and $\|f\|_{L^n(B_1)}$. 
\end{lem}

\begin{proof}
We will prove that there exist a sequence $\{a_k\}$ and some positive $c$ and $b$ such that $\sup u \leq a_k \leq cb^k$ in the sets $A_k = \{x \in \Omega \cap B_{1-2^{-k}} : d(x,\partial\Omega) > 2^{-k}\}$ for all $k \geq 3$. This means roughly that $\sup u$ grows at most like $d^{-K}$ for some big $K$, and then $u^p$ will be integrable if $p > 0$ is small enough.

First, by Lemma \ref{lem:Harnack_interior_Lipschitz_rhs} applied with $\delta = 1/8$, together with the fact that $u \leq 1$ in at least a point of $A_3$, $\sup\limits_{A_3} \leq C(1+\|f\|_{L^n(B_1)}) =: a_3$.

Now, we will show that $a_{k+1} \leq c_1a_k + c_2$, with $c_i > 0$. This easily implies by induction that $a_k \leq cb^k$ for some $b, c > 0$.

Take $x \in A_{k+1}$. We will prove that there exists a close $y \in A_k$ such that\linebreak $d(x,y) < 2^{-k+3}$. In fact, let $z = (z',z_n)$ be the intersection of $\partial A_k$ with the segment $\overline{(e_n/2)x}$. Proving $d(x,z) < 2^{-k+3}$ will suffice, because there are points in $A_k$ arbitrarily close to $z$.

We can parametrise the segment as $\psi(t)=t(e_n/2) + (1-t)x$, with $t \in (0,1)$. Then, $z = \psi(t_*)$, where we define
$$t_* := \inf I_k = \inf\{t \in (0,1) : \psi(t) \in A_k\}.$$

Since $A_k$ is star-shaped with rays coming from $e_n/2$ by Lemma \ref{lem:A_is_convex}, and it contains an open ball around $e_n/2$, $I_k$ is an open interval. Looking closely at the definition of $A_k$, we can write $I_k$ as the intersection of two conditions:
$$I_k = (t_1, 1) \cap (t_2, 1) := \{t \in (0,1) : \psi(t) \in B_{1-2^{-k}}\} \cap \{t \in (0,1) : d(\psi(t), \partial\Omega) > 2^{-k}\}.$$

First, the condition $\psi(t) \in B_{1-2^{-k}}$ means $|te_n/2+(1-t)x| < 1-2^{-k}$, which is automatically fulfilled when $t \geq 2^{-k+1}$, because then
$$|te_n/2+(1-t)x| \leq t/2 + (1-t)|x| < t/2 + (1-t) = 1 - t/2.$$

Hence $t_1 \leq 2^{-k+1}$. To finish this argument we need an upper bound on $t_2$ as well. Take an arbitrary $t \in [2^{-k+2},1]$, and we will see that $d(\psi(t),\partial\Omega) > 2^{-k}$. To do so, we will prove that $\psi_n(t) > g(\psi'(t)) + 2^{-k}\sqrt{L^2+1}$, with $\psi(t) = (\psi'(t), \psi_n(t))$ as usual. Since $g$ is Lipschitz with constant $L < 1/16$, $|g(x')| < 1/16$, and 
$$g(\psi'(t)) \leq g(x') + L|x'-\psi'(t)| < g(x') + t/16,$$
we deduce
\begin{equation*}
\begin{split}
\psi_n(t) &= t/2 + (1-t)x_n > t/2 + (1-t)g(x') = g(x') + t(1/2-g(x'))\\
&> g(x') + 7t/16,\\
\end{split}
\end{equation*}
and
$$g(\psi'(t)) + 2^{-k}\sqrt{L^2+1} < g(x') + t/16 + 2^{-k}\sqrt{L^2+1} < g(x') + t/16 + 3\cdot2^{-k-1}.$$
Finally, since $6t/16 \geq 3\cdot2^{-k-1}$, $\psi_n(t) > g(\psi'(t)) + 2^{-k}\sqrt{L^2+1}$ as desired and $t_2 \leq 2^{-k+2}$. Now, $t_* = \max\{t_1,t_2\} \leq 2^{-k+2}$, and this implies $d(x, z) = t_*d(x,e_n/2) < 2t_* \leq 2^{-k+3}$.

Now, for a given $x \in A_{k+1}$, we have $y \in A_k$ such that $d(x,y) < 2^{-k+3}$. Consider a uniform partition in 32 pieces of the segment $\overline{xy}$, $p_0 = y, \ldots, p_{31} = x$. Since $A_{k+1}$ is star-shaped, $\overline{xy} \subset A_{k+1}$, so the balls $B_i = B_{2^{-k-1}}(p_i)$ are completely contained in $\Omega \cap B_1$. Now, $d(p_i, p_{i+1}) < 2^{-k-2}$, and applying the interior Harnack inequality we get $u(p_{i+1}) \leq C(u(p_i) + \|f\|_{L^n(B_i)}) \leq C(u(p_i) + \|f\|_{L^n(B_1)})$. Iterating this inequality, $u(y) \leq c_1u(x) + c_2$, for some constants $c_1, c_2$ only depending on the dimension and~$\|f\|_{L^n(B_1)}$.

Now we know that $\sup u \leq cb^k$ in $A_k$. Let $p = \log_b \sqrt{2}$, and compute the $L^p$ norm of $u$:
\begin{equation*}
\begin{split}
\int_{B_1}|u|^p &= \int_{A_3}|u|^p + \sum\limits_{j = 3}^\infty\int_{A_{j+1} \setminus A_j}|u|^p \leq |A_3|cb^{3p} + \sum\limits_{j = 4}^\infty|A_j \setminus A_{j - 1}|c(b^p)^j\\
&\leq c\left(2\sqrt{2}|A_3| + \sum\limits_{j = 3}^\infty|A_{j+1} \setminus A_j|2^{j/2}\right) \leq c\left(2\sqrt{2}|B_1| + \sum\limits_{j = 3}^\infty2^{-j}V(n)2^{j/2}\right)\\
&= c\left(2\sqrt{2}|B_1| + V(n)(1+\sqrt{2})/2\right) =: C_p^p,
\end{split}
\end{equation*}
where we have used that $|A_{j+1} \setminus A_j| \leq 2^{-j}V(n)$. We will prove it now.
$$A_{j+1} \setminus A_j \subset (B_1 \setminus B_{1-2^{-j}}) \cup \{x \in B_1 : d(x, \partial\Omega) \leq 2^{-j}\}.$$

On the one hand, $|B_1 \setminus B_{1-2^{-j}}| \leq 2^{-j}|\partial B_1|$, where $|\partial B_1|$ is the $(n-1)$-dimensional measure of the boundary of the ball $B_1$. On the other hand, the second set is a subset of the \textit{thickening} of $\partial\Omega$ in the $e_n$ direction, with height $2^{-j}\sqrt{L^2+1}$ at each side:
$$\left\{(x',x_n) \in B_1' \times \R : |x_n - g(x')| \leq 2^{-j}\sqrt{L^2+1}\right\}.$$

The measure of this second set is $2^{-j+1}\sqrt{L^2+1}|B_1'|$ (again using the measure of $\R^{n-1}$). Hence, defining $V(n) = |\partial B_1|+2\sqrt{1/16^2+1}|B_1'|$ serves our purpose.
\end{proof}

The previous Lemma \ref{lem:BH_rhs_positiva_implica_lp} implies that $u \in L^p(B_1)$. Then, we can use Theorem \ref{thm:mig_Harnack_sup_norma} to obtain the following $L^\infty$ bound on $u$:

\begin{prop}\label{prop:BH_rhs_fitada}
Let $\Omega$ be a Lipschitz domain as in Definition \ref{defn:lipschitz_domain}, $\L$ as in (\ref{eq:non-divergence_operator}) or (\ref{eq:divergence_operator}) and $r \in (0,1)$. Let $u$ be a $L^n$-viscosity or weak solution of
\begin{equation*}
\left\{
\begin{array}{rcll}
\mathcal{L}u & = & f & \text{in } \Omega \cap B_1\\
u & = & 0 & \text{on } \partial\Omega \cap B_1,
\end{array}
\right.
\end{equation*}
with $f \in L^n(\Omega)$.

Then, for all $p > 0$, if $u \in L^p(\Omega\cap B_1)$, $u$ is bounded in $B_r$, i.e.
$$\sup\limits_{B_r} u \leq K(\|u\|_{L^p(B_1)}+ \|f\|_{L^n(\Omega)}),$$
with $K = K(n, p, \lambda, \Lambda, r)$.
\end{prop}

\begin{proof}
Denote $v = u^+$, extending $v$ by zero in $B_1 \setminus \Omega$, and extend $f$ by zero in $B_1 \setminus \Omega$. Then, it is easy to check that $\mathcal{L}v \geq f$ in $B_1$. Now use Theorem \ref{thm:mig_Harnack_sup_norma} and a covering argument to get
$$\sup\limits_{B_r} v \leq C_p(r)(\|u\|_{L^p(B_1)}+ \|f\|_{L^n(B_1)}).$$
The conclusion trivially follows.
\end{proof}

\subsection{Lower bound}

The next step is to construct an iteration to see that solutions of $\mathcal{L}u = f$ that are sufficiently positive away from the boundary, and not very negative near it, are actually positive everywhere. As we have a right hand side $f \in L^q$, we need to be careful with the scaling, so we cannot use directly interior Harnack estimates to prove positivity, and we will need the nondegeneracy estimates in  Lemmas \ref{lem:BH_subsolution} and \ref{lem:div_subsolution}.

\begin{lem}\label{lem:BH_rhs_iteracio}
Let $q > n$, $\kappa > 1$, and let $\mathcal{L}$ be as in (\ref{eq:non-divergence_operator}) or (\ref{eq:divergence_operator}). There exists $L_* = L_*(q,n,\kappa,\lambda,\Lambda)$ such that the following holds.

Let $\Omega$ be a Lipschitz domain as in Definition \ref{defn:lipschitz_domain} with constant $L < L_*$. Let $f$ be such that $\|f\|_{L^q(B_1)} \leq c_0$. Let $u$ be a $L^n$-viscosity or weak solution of
\begin{equation}\label{eq:condicions_iteracio_rhs}
\left\{
\begin{array}{rcll}
\mathcal{L}u & = & f & \text{in } \Omega \cap B_1\\
u & = & 0 & \text{on } \partial\Omega,
\end{array}
\right.
\text{with }
\left\{
\begin{array}{rcll}
u & \geq & 1 & \text{in } \Omega \cap \{x \in B_1 : d(x, \partial\Omega) > \delta\}\\
u & \geq & -\varepsilon & \text{in } \Omega \cap B_1.
\end{array}
\right.
\end{equation}
Then,
\begin{equation*}
\left\{
\begin{array}{rcll}
u & \geq & \rho^\kappa & \text{in } \Omega \cap \{x \in B_\rho : d(x, \partial\Omega) > \rho\delta\}\\
u & \geq & -\rho^\kappa\varepsilon & \text{in } \Omega \cap B_\rho
\end{array}
\right.
\end{equation*}
for some sufficiently small $\rho, \varepsilon, \delta, c_0 \in (0, 1)$, with $\rho > 2\delta$, only depending on the dimension, $\kappa$, $q$, $\lambda$, $\Lambda$, as well as $\sigma$, when applicable.
\end{lem}

\begin{proof}
Let $\beta \in (1, \kappa)$. Apply Lemma \ref{lem:BH_subsolution} in the non-divergence case (respectively, Lemma \ref{lem:div_subsolution} in the divergence case) to obtain $\eta > 0$ (resp. $\eta'$).

Let $h > 0$ to be chosen later. For $x_0 = (x_0', x_{0n})$ in $\{x \in B_\rho : d(x, \partial\Omega) > \rho\delta\}$, define the cone
$$\mathcal{C} = (x_0', g(x_0')) + hC_\eta = \{x \in \R^n : \eta|x' - x_0'| < x_n - g(x_0') < h\}.$$

Here, we distinguish the upper and the lateral boundaries, respectively,
\begin{equation*}
\begin{split}
\partial_u \mathcal{C} &= \{x \in \R^n : \eta|x' - x_0'| \leq x_n = g(x_0') + h\}\\
\partial_l \mathcal{C} &= \{x \in \R^n : \eta|x' - x_0'| = x_n < h\}.
\end{split}
\end{equation*}
and the upper half cone
$$\mathcal{C}^+ = \mathcal{C} \cap \{x_n > g(x_0') + h/2\}.$$

Now, take $L_* = \min\{\eta/2, 1/16\}$. Hence, the slope of $\partial\Omega$ will be at most half of the slope of $\mathcal{C}$, so the cone separates from the boundary. By some geometric computations, we find $d(\partial_u \mathcal{C}, \partial\Omega) \geq h/\sqrt{4+\eta^2}$.

Let now $\rho \leq 1/2$. The distance of the furthest points of $\mathcal{C}$ to $(x_0', g(x_0)')$ is $h\sqrt{1+1/\eta^2}$. Hence, taking $h \leq 1/(2\sqrt{1+1/\eta^2})$ suffices to have $\mathcal{C} \subset \Omega \cap B_1$. Furthermore, making $h = 4\delta\sqrt{4+\eta^2}$, we will have $\partial_u \mathcal{C} \subset \{x \in B_1 : d(x, \partial\Omega) > \delta\}$, and also $\mathcal{C}^+ \subset \{x \in B_1 : d(x, \partial\Omega) > \delta\}$. Note that this forces $\delta$ to be small, but we will choose it at the end, so this is not a problem.

Define $\tilde{u}(x) = u(x_0', g(x_0') + hx) + \varepsilon$. Let $\tilde{u} = v + w$, where

\[
\left\{
\begin{array}{rcll}
\mathcal{L}v & = & 0 & \text{in } C_\eta\\
v & = & \tilde{u} & \text{on } \p C_\eta
\end{array}
\right.
\quad and\quad
\left\{
\begin{array}{rcll}
\mathcal{L}w & = & f & \text{in } C_\eta\\
w & = & 0 & \text{on } \p C_\eta.
\end{array}
\right.
\]

By the ABP estimate, Theorem \ref{thm:ABP}, in the non-divergence form case, or by Theorem \ref{thm:ABP_divergence} in the divergence form case, $\|w\|_{L^\infty(C_\eta)} \leq C'\|f\|_{L^n(C_\eta)} \leq C'c_0$. On the other hand, $v \geq 0$ on $\partial C_\eta$, and $v \geq 1$ on $\partial_u C_\eta$ and $C_\eta^+$ (defining the upper boundary and the upper half anologously). Hence, we can apply Lemma \ref{lem:BH_subsolution} or a rescaled Lemma \ref{lem:div_subsolution} to $v$ to conclude that $v(te_n) \geq t^\beta$, possibly only for small $t < t_\sigma$. 

Putting all together, $\tilde{u}(te_n) \geq t^\beta - C'c_0$, which means 
$$u((x_0', g(x_0')) + hte_n) \geq t^\beta - C'c_0 - \varepsilon,$$
only when $t < t_\sigma$ for divergence form operators. Therefore,
$$u(x_0) \geq \left(\frac{x_{0n} - g(x_0')}{h}\right)^\beta - C'c_0 - \varepsilon \geq \left(\frac{\rho\delta}{h}\right)^\beta - (C'c_0 + \varepsilon) = \left(\frac{\rho}{2\sqrt{4+\eta^2}}\right)^\beta - C'c_0-\varepsilon.$$

Finally, since $\beta < \kappa$, we can choose $\rho > 0$ small enough such that $\rho^{\beta - \kappa} \geq 6\sqrt{4+\eta^2}$ (and $t < t_\sigma$ if needed), and then, choosing $\varepsilon, c_0 > 0$ small enough,
$$u(x_0) \geq 3\rho^\kappa - C'c_0 - \varepsilon \geq \rho^\kappa.$$

Now, for the second inequality, let $x_0 \in B_{1-{3\delta}}$, $d(x_0, \partial\Omega) \leq \delta$. Let $v = u^-/\varepsilon$ in the ball $B_{3\delta(x_0)}$, extending $u$ by $0$ below $\partial\Omega$. By elementary properties of \mbox{$L^n$-viscosity} and weak solutions, since $\mathcal{L}u \leq f$, $\mathcal{L}v \geq -f^+/\varepsilon$. Now, $v \geq 0$ in the whole ball, $v \leq 1$ because $u \geq -\varepsilon$, and $v = 0$ below $\partial\Omega$. Let $z \in \partial\Omega$ be the closest point of the boundary to $x_0$. Let $C_z$ be the downwards cone with slope $L_*$ and vertex in $z$. Then, $C_z$ lies entirely below $\partial\Omega$, and $v = 0$ in $C_z \cap B_{3\delta}(x_0)$. Since $d(x_0, z) \leq \delta$, $|C_z \cap B_{3\delta}(x_0)| \geq c(L_*)|B_{3\delta}(x_0)|$, where $c(L_*)$ is a geometric constant that only depending on the dimension and $L_*$.

Applying Theorem \ref{cor:measure_Harnack_rhs}, $v \leq 1 - \gamma$ in $B_{3\delta/2}(x_0)$, and in particular $v(x_0) \leq 1 - \gamma$. In order to do it, we need $f^+/\varepsilon$ to be small enough, to have $\|f^+/\varepsilon\|_{L^n(B_1)} \leq \delta(c_L)$ in the notation of the theorem.

We will iterate this reasoning with the functions $v_j = v / (1 - \gamma)^j$, defined in $B_{3\delta}(x_0)$, with $x_0 \in B_{1 - 3j\delta}$, $d(x_0, \partial\Omega) \leq \delta$. The conclusion of each iteration is $v_j \leq 1 - \gamma$ in $B_{1 - 3(j+1)\delta}$, i.e.\ $v_{j+1} \leq 1$ in $B_{1 - 3(j+1)\delta}$. This implies $v \leq (1 - \gamma)^j$ in $B_{1 - 3j\delta}$, and then $u \geq -(1-\gamma)^j\varepsilon$ in $B_{1-3j\delta}$.

To end the proof we only need to choose $j$ such that $(1 - \gamma)^j < \rho^\kappa$, and then make $\delta$ small until $1 - 3j\delta \geq \rho$. Finally, notice that we need $\|f^+/((1-\gamma)^i\varepsilon)\|_{L^n(B_1)} \leq \delta(c_L)$ for $i = 1,\ldots,j$ to be able to apply successively Theorem \ref{cor:measure_Harnack_rhs}. This is possible choosing $c_0$ accordingly once we know $j$.
\end{proof}

Now, we iterate the lemma to obtain the desired result.

\begin{prop}\label{prop:BH_rhs_positiva}
Let $q > n$, $\kappa > 1$, and let $\mathcal{L}$ be as in (\ref{eq:non-divergence_operator}) or (\ref{eq:divergence_operator}). There exist $L_* = L_*(q,n,\kappa,\lambda,\Lambda) > 0$ and $\varepsilon, \delta, c_0 \in (0,1)$, such that the following holds. 

Let $\Omega$ be a Lipschitz domain as in Definition \ref{defn:lipschitz_domain} with constant $L < L_*$. Let $u$ be a solution of (\ref{eq:condicions_iteracio_rhs}) with $f$ such that $\|f\|_{L^q(B_1)} \leq c_0$. Then,
$$u > 0 \quad \text{in } \Omega \cap B_{2/3}.$$
Moreover, for all $t \in (0,1)$,
$$u(te_n) \geq t^\kappa.$$
The constants $L_*$, $\varepsilon$, $\delta$ and $c_0$ depend only on the dimension, $\kappa$, $q$, $\lambda$, $\Lambda$, as well as $\sigma$, when applicable.
\end{prop}

\begin{proof}
We will iterate the previous Lemma \ref{lem:BH_rhs_iteracio}. Assume without loss of generality that $\kappa < 2 - n/q$. Let $u_0 = u$, $f_0 = f$, and define the scalings:
$$u_{j + 1}(x) = \rho^{-\kappa}u_j(\rho x), \quad f_{j + 1}(x) = \rho^{2 - \kappa}f_j(\rho x).$$

Define $\Omega_j$ to be the rescaled domains of the $u_j$. Observe that the Lipschitz constant of the domains is the same or smaller, and that $\mathcal{L}u_j = f_j$. Now we will see that the right hand side is bounded as we need. Indeed, since $\rho^{2-\kappa} < \rho^{n/q}$,
$$\|f_{j + 1}\|_{L^q(B_1)} < \left(\int_{B_1}\rho^n|f_j(\rho x)|^q\mathrm{d}x\right)^{1/q} =  \left(\int_{B_\rho}|f_j(y)|^q\mathrm{d}y\right)^{1/q} = \|f_j\|_{L^q(B_\rho)},$$
and then $\|f_j\|_{L^q(B_1)} \leq c_0$ for all $j$.

We will prove by induction that $u_j$ satisfies (\ref{eq:condicions_iteracio_rhs}) for all $j$ as well. Start supposing $u_j$ does. Then, by Lemma \ref{lem:BH_rhs_iteracio}, $u_j \geq \rho^\kappa$ in $\Omega_j \cap \{x \in B_\rho : d(x, \partial\Omega_j) > \rho\delta\}$, which is equivalent to $u_{j + 1} \geq 1$ in $\Omega_{j+1} \cap \{x \in B_\rho : d(x, \partial\Omega_{j+1}) > \rho\delta\}$. Also by the lemma, $u_j \geq -\rho^\kappa\varepsilon$ in $B_\rho$, which is the same as $u_{j + 1} \geq -\varepsilon$ in $B_1$.

All iterates $u_j$ satisfy (\ref{eq:condicions_iteracio_rhs}), thus in particular $u_j(te_n) \geq 1$ for $t \in (2\delta,1)$, taking into account that, since $L_* < \sqrt{3}$, $d(2\delta e_n, \partial\Omega) > \delta$. Rescaling back, this translates easily into $u(te_n) \geq t^\kappa$.

Now, observe that, after a change of variables, choosing smaller $\varepsilon$, $\delta$ and $c_0$ if needed, the function $\tilde{u}(x) = u(x_0 + x/3)$ is also a solution of (\ref{eq:condicions_iteracio_rhs}) for any ${x_0 \in \p\Omega \cap B_{2/3}}$. Analogously, we have $\tilde{u}(te_n) \geq t^\kappa$, thus $u(x_0 + te_n/3) > 0$, and since $\delta < 1/3$ this implies $u > 0$ in ${\Omega \cap B_{2/3}}$.
\end{proof}

As a consequence, we find:

\begin{cor}\label{cor:BH_rhs_growth}
Let $q > n$, $\kappa > 1$, and let $\mathcal{L}$ be as in (\ref{eq:non-divergence_operator}) or (\ref{eq:divergence_operator}). There exists $L_* = L_*(q, n, \kappa, \lambda, \Lambda)$ such that the following holds.

Let $\Omega$ be a Lipschitz domain as in Definition \ref{defn:lipschitz_domain} with constant $L < L_*$. Let $f$ such that $\|f\|_{L^q(B_1)} \leq c_0$. Let $v$ be a positive $L^n$-viscosity or weak solution of $\mathcal{L}v = f$, with $v(e_n/2) \geq 1$. Then, for all $x \in B_{1/2}$,
$$v(x) \geq c_1 d(x,\p\Omega)^\kappa,$$
for some sufficiently small $c_0, c_1 > 0$, only depending on the dimension, $\kappa$, $q$, $\lambda$, $\Lambda$, as well as $\sigma$, when applicable.
\end{cor}

\begin{obs}\label{obs:BH_rhs_growth}
We can also write $x_n - g(x')$ instead of $d(x,\p\Omega)$, since the two quantities are comparable.
\end{obs}

\begin{proof}
Assume, after dividing by a constant if necessary, that $v(e_n/2) = 1$. Let $v'(x) = v(2x)$ and $\Omega'$ the corresponding scaled domain. By a version of Lemma \ref{lem:Harnack_interior_Lipschitz_rhs}, we have that if $c_0$ is small enough, $v' \geq c_2 > 0$ in $\{x \in \Omega' \cap B_{3/2} : d(x, \partial\Omega') > \delta\}$, for any $\delta > 0$, some small $c_2 > 0$ that depends only on $\delta$, the dimension, $q$ and the ellipticity constants. Now, apply the previous Proposition \ref{prop:BH_rhs_positiva} to $v'/c_2$ in the balls $B_1(x_0)$ for any $x_0 \in B_1$ (we may need to ask that $c_0$ is smaller to do so). Hence, $v'(x_0 + te_n)/c_2 \geq t^\kappa$, and this implies $v(x_0 + (t/2)e_n) \geq c_2 (t/2)^\kappa$. To end the proof, notice that $d(x_0 + (t/2)e_n, \partial\Omega) \in [t/2, t\sqrt{L^2+1}/2]$, so we can absorb the factor needed to change $t$ for $d(x,\partial\Omega)$ in the constant $c_1$.
\end{proof}

\subsection{Proof of the main result}
Now we have all that we need to prove Theorem \ref{thm:BH_rhs}. Observe that Corollary \ref{cor:BH_rhs_symmetric} is a direct consequence. We divide the proof in two parts: in the first one we prove the inequality, and in the second we deduce the $\mathcal{C}^{0,\alpha}$ regularity of $u/v$.

\begin{proof}[Proof of Theorem \ref{thm:BH_rhs}]
We prove the inequality first. Let
$$\kappa = 1 + \frac{1}{2}\left(1 - \frac{n}{q}\right) > 1,$$
and choose $L_0(q,n,\lambda,\Lambda) = L_*(q,n,\kappa,\lambda,\Lambda)$ with the definition of $L_*$ given by Proposition \ref{prop:BH_rhs_positiva}. We will still keep $\kappa$ explicit to simplify some calculations. If we are in the case $u > 0$, apply Lemma \ref{lem:BH_rhs_positiva_implica_lp}. In either case, by Proposition \ref{prop:BH_rhs_fitada}, $u \leq K$ in $B_{3/4}$.

Then, consider $v$ in the set $A = \{x \in \overline{B_{3/4}} : d(x, \partial\Omega) \geq 3\delta/4\}$. $A$ is a subset of $\{x \in \overline{B_{1-3\delta/4}} : d(x, \partial\Omega) \geq 3\delta/4\}$. Hence, by Lemma \ref{lem:Harnack_interior_Lipschitz_rhs}, and from $v(e_n/2) \geq 1$, it follows that $v \geq C^{-1} - \|f\|_{L^n(B_1)} \geq C^{-1} - c_0$ in the whole set $A$. Furthermore, choosing $m = C^{-1}/2$ and $c_0 \leq m$ yields $v \geq m > 0$ in $A$.

Define now
$$w := \frac{1+\varepsilon}{m}v - \frac{\varepsilon}{K}u,$$
with $\varepsilon > 0$ to be determined later. We will show that $w > 0$ in $B_{1/2}$, and therefore, taking $C = K(1+\varepsilon)/(m\varepsilon)$, $Cv - u > 0$.

By construction, $w \geq v/m \geq 1$ in $A$, and $w \geq -\varepsilon$ in $B_{3/4}$. To apply Proposition \ref{prop:BH_rhs_positiva} (rescaled to the ball $B_{3/4}$), we need to estimate $\mathcal{L}w$:
$$\|\mathcal{L}w\|_{L^q(B_{3/4})} \leq \frac{1+\varepsilon}{m}\|\mathcal{L}v\|_{L^q(B_1)} + \frac{\varepsilon}{K}\|\mathcal{L}u\|_{L^q(B_1)} \leq \left(\frac{1+\varepsilon}{m} + \frac{\varepsilon}{K}\right)c_0.$$

Let $\tilde{w}(x) = w(3x/4)$. Then, $\tilde{w} \geq 1$ in $\Omega \cap \{x \in B_1 : d(x, \partial\Omega) \geq \delta\}$ and $\tilde{w} \geq -\varepsilon$ in $\Omega \cap B_1$. Choosing sufficiently small $\varepsilon, c_0 > 0$ to apply Proposition \ref{prop:BH_rhs_positiva}, we get $\tilde{w} > 0$ in $B_{2/3}$, thus $w > 0$ in $B_{1/2}$.

Now, for the boundary $\mathcal{C}^{0,\alpha}$ regularity of the quotient $u/v$, we will first prove the regularity for the boundary points, and then we will extend it to the whole closed domain $\overline{\Omega \cap B_{1/2}}$, where $u/v$ is extended by continuity on $\partial\Omega$. These arguments are the standard ones found in the literature, but we have to be careful with some calculations to take into account the right hand side of the equations. Additionally, let $c_0^*$ be the value for $c_0$ found in the first part of the proof. We will adjust the final value of $c_0$ in terms of this $c_0^*$.

By a covering argument, the inequality $u \leq C'v$ is valid in $\Omega \cap B_{3/4}$ with an appropriate constant $C'$. Since either $u > 0$ or we can interchange $u$ by $-u$ and the hypotheses still hold, we have $u \geq -C'v$ as well. Let $x_0 \in \partial\Omega \cap \overline{B_{1/2}}$.

First, we will show by induction that there exist sequences $\{a_j\}, \{b_j\}$ such that, for every integer $j \geq 2$,
$$a_j v \leq u \leq b_j v \ \, \text{in} \ \, \Omega \cap B_{2^{-j}}(x_0), \ \, (b_{j+1}-a_{j+1}) = (1 - \theta)(b_j - a_j), \ \, \theta \in (0, 1- 2^{1-\kappa}].$$

For $j = 2$ we take $a_j = -C', b_j = C'$, with the constant from the covering argument. Now, to perform the inductive step, we define two new functions:
$$w_1 := \frac{u - a_j v}{b_j - a_j}, \quad w_2 := \frac{b_j v - u}{b_j - a_j}.$$

These functions are positive solutions of $\mathcal{L}w_i = f_i$ in $\Omega \cap B_{2^{-j}}(x_0)$, vanish continuously at $\partial\Omega$, and $w_1 + w_2 = v$. Therefore, for one of them (the biggest in the point), $2w_i(x_0 + e_n / 2^{j + 1}) \geq v(x_0 + e_n / 2^{j + 1})$. To apply the boundary Harnack, we define the following rescaled functions, with $c_1 > 0$ from Corollary \ref{cor:BH_rhs_growth} in order to have $\tilde{v}(e_n/2) \geq 1$. Let
\begin{gather*}
\tilde{v}(x) = c_1^{-1}2^{j\kappa}v(x_0 + 2^{-j}x),\quad \tilde{w}_i(x) = c_1^{-1}2^{j\kappa}w_i(x_0 + 2^{-j}x),\\
\tilde{f}_1(x) = 2^{j(\kappa-2)}\frac{f(x_0 + 2^{-j}x) - a_jg(x_0 + 2^{-j}x)}{c_1(b_j - a_j)},\\
\tilde{f}_2(x) = 2^{j(\kappa-2)}\frac{b_jg(x_0 + 2^{-j}x) - f(x_0 + 2^{-j}x)}{c_1(b_j - a_j)}.
\end{gather*}

Now we must check $\|\tilde{f}_i\|_{L^q(B_1)} \leq c_0^*$. Indeed, choosing $c_0$ appropriately,
\begin{equation*}
\begin{split}
\|\tilde{f}_1\|_{L^q(B_1)} &\leq \frac{\|2^{j(\kappa-2)}f(x_0 + 2^{-j}x)\|_{L^q(B_1)} + a_j\|2^{j(\kappa-2)}g(x_0 + 2^{-j}x)\|_{L^q(B_1)}}{c_1(b_j - a_j)}\\
&\leq c_0\frac{2^{j(n/q + \kappa-2)}(1+|a_j|)}{c_1(b_j - a_j)} \leq \frac{2^{j(n/q + \kappa-2)}c_0}{c_1 (1 - \theta)^{j-2}} \leq \frac{c_0(1 - \theta)^2}{c_1} \leq c_0^*.
\end{split}
\end{equation*}
The same works for $f_2$.

Applying a rescaled version of the boundary Harnack inequality to the functions $2w_i, v$, we get that $w_i \geq \frac{v}{2C'}$ in $\Omega \cap B_{2^{-(j+1)}}(x_0)$. This presents two options: either
$$\frac{u - a_j v}{b_j - a_j} \geq \frac{v}{2C'} \quad \Rightarrow \quad u \geq \left(a_j + \frac{b_j - a_j}{2C'}\right)v =: \tilde{a}_{j+1}v, \quad \tilde{b}_{j+1} = b_j,$$
or
$$\frac{b_j v - u}{b_j - a_j} \geq \frac{v}{2C'} \quad \Rightarrow \quad u \leq \left(b_j - \frac{b_j - a_j}{2C'}\right)v =: \tilde{b}_{j+1}v, \quad \tilde{a}_{j+1} = a_j.$$

Either $\tilde{a}_{j + 1} > a_j$ or $\tilde{b}_{j + 1} < b_j$. We cannot choose them yet as $a_{j + 1}, b_{j + 1}$, because we need to ensure $1 - \theta \geq 2^{1-\kappa}$. This is done by choosing
\begin{align*}
&a_{j + 1} = \min\{\tilde{a}_{j+1}, a_j + (1-2^{1-\kappa})(b_j - a_j)\},\\
&b_{j + 1} = \max\{\tilde{b}_{j+1}, b_j - (1-2^{1-\kappa})(b_j - a_j)\}.
\end{align*}

After this, $a_j \leq u/v \leq b_j$ in $\Omega \cap B_{2^{-j}}(x_0)$, and then
\begin{equation}\label{eq:BH_rhs_ball_estimate}
\sup\limits_{B_{2^{-j}(x_0)\cap\Omega}}u/v - \inf\limits_{B_{2^{-j}(x_0)\cap\Omega}}u/v \leq b_j - a_j = (2C')(1 - \theta)^{j-2}.
\end{equation}

We can extend $u/v$ by continuity at $x_0$ as the limit of the $a_j$ (or the $b_j$), and for any point $p \in \Omega \cap B_{2^{-j}}(x_0)$, $|(u/v)(x_0) - (u/v)(p)| \leq 2C'(1 - \theta)^{j-2}$, hence $u/v$ is $\mathcal{C}^{0,\alpha}$ at $x_0$ with $\alpha = -\log_2 (1 - \theta)$. Then, for every point $x_0$ on the boundary we have 
$$\left|\frac{u}{v}(x_0) - \frac{u}{v}(p)\right| \leq C|x_0 - p|^\alpha,$$
for some uniform constant $C > 0$, for any $p \in B_{1/2} \cap \overline{\Omega}$.

Now, for the interior points, let $x_1, x_2 \in B_{1/2}$, $d = |x_1 - x_2|$ and $\delta_i = d(x_i, \partial\Omega)$. There are three different cases:

\emph{Case 1}. If $d \geq 1/16$, we just use the fact that $-C'v \leq u \leq C'v$ in $\Omega \cap B_{3/4}$, hence, for any $\alpha \in (0, 1)$,
$$\left|\frac{u}{v}(x_1) - \frac{u}{v}(x_2)\right| \leq 2C' \leq 32C'|x_1 - x_2|^\alpha.$$

\emph{Case 2}. If the points are far compared with the distance to the boundary, in the sense that $d \geq \delta_i/4$ for at least one of them, let $y$ be a point in the boundary such that $d(x_i, y) < 8d$ for both of them (for example, in the case $\delta_1 \leq 4d$, let $y$ be the closest point in the boundary to $x_1$, so that $d(x_2, y) \leq \delta_1 + d \leq 5d$). Then,
\begin{align*}
\left|\frac{u}{v}(x_1) - \frac{u}{v}(x_2)\right| &\leq \left|\frac{u}{v}(x_1) - \frac{u}{v}(y)\right| + \left|\frac{u}{v}(y) - \frac{u}{v}(x_2)\right|\\
&\leq C(|x_1 - y|^\alpha + |y - x_2|^\alpha) \leq 2C(8d)^\alpha \leq 2^{1+3\alpha}C|x_1 - x_2|^\alpha.
\end{align*}

\emph{Case 3}. When the points are close, i.e.\ $d < 1/16$ and $d < \min(\delta_1, \delta_2)/4$, suppose without loss of generality $0 < \delta_1 \leq \delta_2$. Let 
$$r = d(x_2, \partial(B_{3/4}\cap\Omega)) = \min\{3/4 - |x_2|, \delta_2\} \geq \min\{1/4,\delta_2\}.$$ Now, we introduce an auxiliary function $w = u - \mu v$, with $\mu$ to be determined later.
\begin{equation*}
\left|\frac{u}{v}(x_1) - \frac{u}{v}(x_2)\right| \leq \frac{v(x_1)|w(x_1) - w(x_2)| + |w(x_2)||v(x_1)-v(x_2)|}{v(x_1)v(x_2)}.
\end{equation*}

Hence, since $\mathcal{L}w = f - \mu g$, $\mathcal{L}v = g$, by interior regularity estimates,
\begin{equation*}
\begin{split}
|w(x_1) - w(x_2)| &\leq C_1|x_1 - x_2|^{\alpha'}\left(r^{-\alpha'}\|w\|_{L^\infty(B_{r/2}(x_2))} + r^{2-n/q-\alpha'}\|\mathcal{L}w\|_{L^q(B_{r/2}(x_2))}\right)\\
&\leq C_1|x_1 - x_2|^{\alpha'}\left(r^{-\alpha'}\|w\|_{L^\infty(B_{r/2}(x_2))} + (1+|\mu|)c_0r^{2-n/q-\alpha'}\right),\\
|v(x_1) - v(x_2)| &\leq C_1|x_1 - x_2|^{\alpha'}\left(r^{-\alpha'}\|v\|_{L^\infty(B_{r/2}(x_2))} + r^{2-n/q-\alpha'}\|g\|_{L^q(B_{r/2}(x_2))}\right)\\
&\leq C_1|x_1 - x_2|^{\alpha'}\left(r^{-\alpha'}\|v\|_{L^\infty(B_{r/2}(x_2))} + c_0r^{2-n/q-\alpha'}\right).
\end{split}
\end{equation*}
We may assume without loss of generality that $\alpha' \in (0, 2 - \kappa)$.

Now, by the interior Harnack inequality and Corollary \ref{cor:BH_rhs_growth}, tweaking the constants, $v(x_i) \leq \|v\|_{L^\infty(B_{r/2}(x_2))} \leq Cv(x_i)$. Then, combining our estimates,
\begin{align*}
\frac{1}{C_1}&\left|\frac{u}{v}(x_1) - \frac{u}{v}(x_2)\right||x_1 - x_2|^{-\alpha'} \leq\\
\leq &\frac{\|w\|_{L^\infty(B_{r/2}(x_2))} + (1+|\mu|)c_0r^{2-n/q}}{v(x_2)r^{\alpha'}} + \frac{|w(x_2)|(\|v\|_{L^\infty(B_{r/2}(x_2))} + c_0r^{2-n/q})}{v(x_1)v(x_2)r^{\alpha'}}\\
\leq &\frac{C\|w\|_{L^\infty(B_{r/2}(x_2))}}{r^{\alpha'}\|v\|_{L^\infty(B_{r/2}(x_2))}} + \frac{C(1 + |\mu|)c_0r^{2-n/q}}{r^{\alpha'}\|v\|_{L^\infty(B_{r/2}(x_2))}} + \frac{C^2\|w\|_{L^\infty(B_{r/2}(x_2))}}{r^{\alpha'}\|v\|_{L^\infty(B_{r/2}(x_2))}}+\\  &+\frac{C^2\|w\|_{L^\infty(B_{r/2}(x_2))}c_0r^{2-n/q}}{r^{\alpha'}\|v\|_{L^\infty(B_{r/2}(x_2))}^2}.
\end{align*}

Now we distinguish two cases: when $r = 1/4$ we just use the global estimates, and when $r < 1/4$ we do some finer computations.

    \emph{Case 3.1}. When $r = 1/4$, let $\mu = 0$. Hence $w = u$. Since $-C'v \leq u \leq C'v$\linebreak in ${B_{3/4}\cap\Omega}$, $\|w\|_{L^\infty(B_{r/2}(x_2))} \leq C'\|v\|_{L^\infty(B_{r/2}(x_2))}$, and $\|v\|_{L^\infty(B_{r/2}(x_2))} \geq c_1 r^\kappa$ by Corollary \ref{cor:BH_rhs_growth}, then the right hand side of the previous inequality is bounded by some constant $C_2$ that only depends on $n, q, \lambda, \Lambda$. Hence,
    $$\left|\frac{u}{v}(x_1) - \frac{u}{v}(x_2)\right| \leq C_1C_2|x_1 - x_2|^{\alpha'}.$$
    
    \emph{Case 3.2}. If $r < 1/4$, $r = \delta_2$. Choose $\mu = u(x_2)/v(x_2)$, so that $w(x_2) = 0$. Let $k_0$ be the maximum positive integer such that $\delta_2 < 2^{-k_0}$ (hence $\delta_2 \geq 2^{-k_0 -1}$). Then, $k_0 \geq 2$, $d < \delta_2/4$, and $x_1, x_2$ belong to $\Omega \cap B_{2^{-k_0 + 1}}(y)$, with $y \in \partial\Omega$, for instance, the closest point in $\partial\Omega$ to $x_2$ ($d(y,x_1) < \delta_2 + d < 2\delta_2 < 2^{-k_0+1}$ by the triangle inequality). For the same reason, $B_r(x_2) \subset \Omega \cap B_{2^{-k_0+1}}(y)$.
    
    By the estimate (\ref{eq:BH_rhs_ball_estimate}), $\|w\|_{L^\infty(B_r(x_2))} \leq (2C')(1 - \theta)^{k_0 - 2}\|v\|_{L^\infty(B_r(x_2))}$, and combining it with the previous result and the fact that $1 - \theta = 1/2^\alpha$,
    \begin{align*}
    \frac{1}{C_1}\left|\frac{u}{v}(x_1) - \frac{u}{v}(x_2)\right||x_1 - x_2|^{-\alpha'} \leq &\frac{2C'(C+C^2)(1 - \theta)^{k_0-2}}{r^{\alpha'}}+\\
    &+ \frac{(C(1+|\mu|) + C^2(2C'(1 - \theta)^{k_0-2}))c_0r^{2-n/q-\alpha'}}{\|v\|_{L^\infty(B_{r/2}(x_2))}}.
    \end{align*}

    We put all the constants (everything that does not depend on $r, k_0$) together, and notice that $|\mu| \leq C'$ and $2^{-k_0-1} \leq r < 2^{-k_0}$. Additionally, we dismiss the term $(1 - \theta)^{k_0-2} \leq 1/(1 - \theta)^2$ as a constant in the second fraction. Simplifying, we get
    $$\left|\frac{u}{v}(x_1) - \frac{u}{v}(x_2)\right| \leq |x_1 - x_2|^{\alpha'}\left(2^{k_0(\alpha' - \alpha)}C_3 + \frac{2^{k_0(n/q+\alpha'-2)}C_4}{\|v\|_{L^\infty(B_{r/2}(x_2))}}\right).$$
    
    Since $\|v\|_{L^\infty(B_{r/2}(x_2))} \geq c_1 r^\kappa = c_1\delta_2^\kappa \geq c_12^{-\kappa(k_0+1)}$,
    $$\frac{2^{k_0(n/q+\alpha'-2)}C_4}{\|v\|_{L^\infty(B_{r/2}(x_2))}} \leq 2^{k_0(n/q+\alpha'+\kappa - 2)+1}C_4/c_1 \leq 2C_4/c_1.$$
    
    If $\alpha \geq \alpha'$, $|x_1-x_2|^{\alpha'}2^{k_0(\alpha'-\alpha)} \leq |x_1-x_2|^{\alpha'}$. If $\alpha < \alpha'$, take into account that $r = |x_1 - x_2| < 2^{-k_0}$, and then
    $$|x_1-x_2|^{\alpha'}2^{k_0(\alpha'-\alpha)} = \left(\frac{|x_1-x_2|}{2^{-k_0}}\right)^{\alpha'}2^{-k_0\alpha} \leq \left(\frac{|x_1-x_2|}{2^{-k_0}}\right)^{\alpha}2^{-k_0\alpha} = |x_1 - x_2|^\alpha.$$
    
    In either case,
    $$|x_1-x_2|^{\alpha'}2^{k_0(\alpha'-\alpha)} \leq |x_1-x_2|^{\min\{\alpha, \alpha'\}}.$$
    
    Hence,
    $$\left|\frac{u}{v}(x_1) - \frac{u}{v}(x_2)\right| \leq C_5|x_1 - x_2|^{\min\{\alpha,\alpha'\}}.$$

Observe that we have proved that $|(u/v)(x_1) - (u/v)(x_2)| \leq C|x_1 - x_2|^\alpha$ for various values of $C, \alpha > 0$. For the expression to be always valid, take the maximum multiplicative constant and the minimum exponent. 
\end{proof}

\section{The boundary Harnack in slit domains}\label{sect:slit}
We also consider our problem in \textit{slit domains}, as introduced in \cite{DS16, DS20}. We define them in the unit ball $B_1$ to keep the notation uncluttered.
\begin{defn}\label{defn:slit_domain}
We say $\Omega$ is a slit domain with Lipschitz constant $L$ if $\Omega = B_1 \setminus K$, with $K$ a closed subset of the graph of a Lipschitz function $g : B_1' \to \R$, with $g(0) = 0$:
$$\Omega = B_1 \setminus K, \quad K \subset \Gamma := \{(x',x_n)\in B_1'\times\R : x_n = g(x')\}, \quad \|g\|_{\mathcal{C}^{0,1}} = L.$$
Additionally, we define the upper and lower \textit{halves} of $\Omega$,
$$\Omega^+ = \Omega \cap \{(x',x_n)\in B_1 : x_n \geq g(x')\}, \quad \Omega^- = \Omega \cap \{(x',x_n)\in B_1 : x_n \leq g(x')\}.$$
We will write $\Omega^\pm$ to refer to $\Omega^+$ or $\Omega^-$ indistinctly.
\end{defn}

An analogous reasoning to the proof of Theorem \ref{thm:BH_rhs} for slit domains yields the following result.

\begin{thm}\label{thm:BH_rhs_slit}
Let $q > n$ and let $\mathcal{L}$ be as in (\ref{eq:non-divergence_operator}) or (\ref{eq:divergence_operator}). There exist small constants $c_0 > 0$ and $L_0 > 0$ such that the following holds.

Let $\Omega = B_1 \setminus K$ be a slit domain as in Definition \ref{defn:slit_domain}, with Lipschitz constant $L < L_0$. Let $u$ and $v > 0$ be $L^n$-viscosity or weak solutions of
\begin{equation*}
\left\{
\begin{array}{rcll}
\mathcal{L}u & = & f & \text{in } B_1 \setminus K\\
u & = & 0 & \text{on } K
\end{array}
\right.
\ \text{and}\quad
\left\{
\begin{array}{rcll}
\mathcal{L}v & = & g & \text{in } B_1 \setminus K\\
v & = & 0 & \text{on } K,
\end{array}
\right.
\end{equation*}
with $f$ and $g$ satisfying (\ref{eq:f_g_fitades}).

Additionally, assume that $v(e_n/2) \geq 1$, $v(-e_n/2) \geq 1$, and either $u > 0$ in $B_1 \setminus K$ and $\max\{u(e_n/2),u(-e_n/2)\} \leq 1$, or $\|u\|_{L^p(B_1)} \leq 1$ for some $p > 0$. Then,
$$u \leq Cv \quad \text{in} \quad B_{1/2} \setminus K,$$
and
$$\left\|\frac{u}{v}\right\|_{\mathcal{C}^{0,\alpha}(\overline{\Omega^\pm}\cap B_{1/2})} \leq C.$$

The positive constants $C$, $c_0$, $L_0$ and $\alpha$ depend only on the dimension, $q$, $\lambda$, $\Lambda$, as well as $p$ and $\sigma$, when applicable.
\end{thm}

When both functions are positive, we recover the symmetric version of the boundary Harnack.
\begin{cor}\label{cor:BH_rhs_slit_symmetric}
Let $q > n$ and $\mathcal{L}$ as in (\ref{eq:non-divergence_operator}) or (\ref{eq:divergence_operator}). There exist small constants $c_0 > 0$ and $L_0 = L_0(q,n,\lambda,\Lambda) > 0$ such that the following holds.

Let $\Omega = B_1 \setminus K$ be a slit domain as in Definition \ref{defn:slit_domain}, with Lipschitz constant $L < L_0$. Let $u, v$ be positive $L^n$-viscosity or weak solutions of
\begin{equation*}
\left\{
\begin{array}{rcll}
\mathcal{L}u & = & f & \text{in } B_1 \setminus K\\
u & = & 0 & \text{on } K
\end{array}
\right.
\ \text{and}\quad
\left\{
\begin{array}{rcll}
\mathcal{L}v & = & g & \text{in } B_1 \setminus K\\
v & = & 0 & \text{on } K,
\end{array}
\right.
\end{equation*}
with $f$ and $g$ satisfying (\ref{eq:f_g_fitades}).

Assume $\min\{v(e_n/2),v(-e_n/2)\} \geq 1$ and $\min\{u(e_n/2),u(-e_n/2)\} \geq 1$. Then,
$$C^{-1}\frac{\min\{u(e_n/2),u(-e_n/2)\}}{\max\{v(e_n/2),v(-e_n/2)\}} \leq \frac{u}{v} \leq C\frac{\max\{u(e_n/2),u(-e_n/2)\}}{\min\{v(e_n/2),v(-e_n/2)\}} \quad \text{in} \quad B_{1/2} \setminus K,$$
and
$$\left\|\frac{u}{v}\right\|_{\mathcal{C}^{0,\alpha}(\overline{\Omega^\pm}\cap B_{1/2})} \leq C.$$

The positive constants $C$, $c_0$, $L_0$ and $\alpha$ depend only on the dimension, $q$, $\lambda$, $\Lambda$, as well as $\sigma$, when applicable.
\end{cor}

Most of the proofs are identical to the one-sided theorem, because we can prove the results for each \textit{side} of $\Gamma$ and then put them together. There are two exceptions: Proposition \ref{prop:BH_rhs_fitada} and Lemma \ref{lem:BH_rhs_iteracio}. The proof of the proposition is even easier, taking $v = u^+$ and extending it by $0$ on $K$, we are ready to apply Theorem \ref{thm:mig_Harnack_sup_norma} and see that $v$ is universally bounded.

As for the lemma, we write here an adapted version and the step of the proof that needs to be changed.
\begin{lem}\label{lem:iteration_slit_domains}
Let $q > n$, $\kappa > 1$, and let $\mathcal{L}$ be as in (\ref{eq:non-divergence_operator}) or (\ref{eq:divergence_operator}). There exists $L_* = L_*(q,n,\kappa,\lambda,\Lambda)$ such that the following holds.

Let $\Omega = B_1 \setminus K$, with $K \subset \Gamma$, be a slit domain as in Definition \ref{defn:slit_domain} with constant $L < L_*$. Let $f$ such that $\|f\|_{L^q(B_1)} \leq c$. Let $u$ be a $L^n$-viscosity or weak solution of
\begin{equation}\label{eq:condicions_iteracio_rhs_slit}
\left\{
\begin{array}{rcll}
\mathcal{L}u & = & f & \text{in } B_1 \setminus K\\
u & = & 0 & \text{on } K,
\end{array}
\right.
\text{with }
\left\{
\begin{array}{rcll}
u & \geq & 1 & \text{in } \Omega \cap \{x \in B_1 : d(x, \Gamma) > \delta\}\\
u & \geq & -\varepsilon & \text{in } \Omega \cap B_1.
\end{array}
\right.
\end{equation}
Then,
\begin{equation*}
\left\{
\begin{array}{rcll}
u & \geq & \rho^\kappa & \text{in } \Omega \cap \{x \in B_\rho : d(x, \Gamma) > \rho\delta\}\\
u & \geq & -\rho^\kappa\varepsilon & \text{in } B_\rho
\end{array}
\right.
\end{equation*}
for some sufficiently small $\rho, \varepsilon, \delta, c \in (0, 1)$, with $\rho > 2\delta$, only depending on the dimension, $\kappa$, $q$, $\lambda$, $\Lambda$, as well as $\sigma$, when applicable.
\end{lem}

\begin{proof}
The proof of the first inequality is completely analogous to the proof of Lemma \ref{lem:BH_rhs_iteracio}.

For the second inequality, we do the same reasoning as in the one-sided case, but now, instead of picking a downwards cone $C_z$ with vertex at $\partial\Omega$, for each $x_0$ such that $d(x_0, \Gamma) \leq \delta$, we take $z = x_0 - 5\delta/2 e_n$. Since $\Gamma$ is a Lipschitz graph with Lipschitz constant $L < 1/16$, $d(z, \Gamma) \geq 5\delta/(2\sqrt{L^2+1}) - \delta > \delta$, so again $z$ and the analogous downwards cone $C_z$ lie in the region where $u \geq \rho^\kappa$. Moreover, $|C_z \cap B_{3\delta}(x_0)| = c_L|B_{3\delta}|$. The rest of the proof continues analogously.
\end{proof}

\section{Applications to free boundary problems}\label{sect:applications}
\subsection{$\mathcal{C}^{1,\alpha}$ regularity of the free boundary in the obstacle problem}
Consider the classical obstacle problem (\ref{eq:classical_obstacle}) in $B_1$, with $f \geq \tau_0 > 0$, $f \in W^{1,q}$, and assume that $0$ is a free boundary point. We will show that we can extend the proof of the $\mathcal{C}^{1,\alpha}$ regularity of the free boundary due to Caffarelli \cite{Caf98} to the case $f \in W^{1,q}$ thanks to our new result. We generalize the steps of the proof in \cite[Section~5.4]{FR20}.

Our starting point will be the existence of a \textit{regular} blow-up. We will also take for granted the following nondegeneracy condition: if $x_0 \in \overline{\{u > 0\}}$,
$$\sup\limits_{B_r(x_0)} u \geq cr^2,$$
which follows easily from the fact $f \geq \tau_0 > 0$; see \cite[Proposition 5.9]{FR20}.

\begin{prop}\label{prop:classic_FB_Lipschitz}
Let $u$ be a solution of (\ref{eq:classical_obstacle}), with $f \in W^{1,n}$ and $f \geq \tau_0 > 0$. Assume that $0$ is a regular free boundary point as in Definition \ref{defn:regular_point_classical}.

Then, for every $L_0 > 0$ there exists $r > 0$ such that the free boundary is the graph of a Lipschitz function in $B_r$ with Lipschitz constant $L < L_0$.
\end{prop}

We will denote
$$u_r(x) := \frac{u(rx)}{r^2}.$$
Observe that the blow-up hypothesis implies that for all $\varepsilon > 0$, there exists $r_0$ such that
$$\left|u_{r_0} - \frac{\gamma}{2}(x\cdot e)_+^2\right| < \varepsilon \quad \text{in} \quad B_1,$$
and
$$\big|\partial_\nu u_{r_0} - \gamma(x\cdot e)_+(x\cdot \nu)\big| < \varepsilon \quad \text{in} \quad B_1$$
for all $\nu \in \mathbb{S}^{n-1}$.

To prove that the free boundary is Lipschitz, we will use the interior and exterior cone conditions, and to do this we will prove that $\partial_\nu u_r \geq 0$, with $\nu$ a unit vector, when $\nu \cdot e > c(L)$, where $c(L)$ is the positive constant that ensures that the cone $\{x \in \R^n : x\cdot e = |x|c(L)\}$ has Lipschitz constant $L$. We need a positivity lemma.

\begin{lem}\label{lem:classical_obstacle_derivada_positiva}
Let $u$ be a solution of (\ref{eq:classical_obstacle}) with $f \in W^{1,n}(B_1)$, $r > 0$ and ${\Omega = \{u_r > 0\}}$. Let $w = \partial_\nu u_r$. Then, $w$ is a solution of
\begin{equation*}
\left\{
\begin{array}{rcll}
\Delta w & = & g & \text{in } \Omega \cap B_1\\
w & = & 0 & \text{on } \partial\Omega,
\end{array}
\right.
\end{equation*}
with $g(x) = r\partial_\nu f (rx)$. Assume that, denoting $N_\delta = \{x \in B_1 : d(x,\partial\Omega) < \delta\}$, we have
\begin{equation}\label{eq:condicions_lema_positivitat}
w > -\varepsilon \quad \text{in} \quad N_\delta \quad \text{and} \quad w > M \quad \text{in} \quad \Omega \setminus N_\delta,
\end{equation}
with positive $\varepsilon$ and $M$. Then, $w \geq 0$ in $\Omega \cap B_{1/2}$, provided that $\varepsilon$, $r$ and $\delta > 0$ are small enough.
\end{lem}

\begin{proof}
First, it is clear that $w > 0$ in $\Omega \setminus N_\delta$. Suppose there exists $x_0 \in B_{1/2} \cap N_\delta$ such that $w(x_0) < 0$. We will arrive at a contradiction using the maximum princple, combined with the ABP estimate, with the function
$$v(x) = w(x) - \eta\left(u_r(x) - \frac{f(x_0)}{2n}|x - x_0|^2\right).$$

Consider the set $\Omega \cap B_{1/4}(x_0)$. On $\partial\Omega$, $u_r = 0$, hence $v \geq 0$. On $\partial B_{1/4}(x_0) \cap N_\delta$,
$$v(x) \geq -\varepsilon - \eta\delta\|u_r\|_{\mathcal{C}^1} + \frac{\eta}{32n}.$$

On $\partial B_{1/4}(x_0) \cap (\Omega \setminus N_\delta)$,
$$v(x) \geq M - \eta\|u_r\|_{\mathcal{C}^1}.$$

Notice that $\|u_r\|_{\mathcal{C}^1}$ is uniformly bounded as $r \rightarrow 0$. Hence, choosing $\eta$ small enough, the second inequality implies $v \geq M/2$. For the first inequality, choosing now small enough $\varepsilon$ and $\delta$, we obtain $v \geq \eta/(64n)$.

This function satisfies $\Delta v (x) = g(x) - \eta(f(x) - f(x_0))$, hence, by the ABP estimate,
$$v(x_0) \geq \min\{M/2,\eta/(64n)\} - C\|g(x) - \eta(f(rx)-f(rx_0))\|_{L^n(B_{1/4}(x_0))}.$$

We estimate $g$ and $f - f(x_0)$ separately. Using the scaling of the $L^n$ norm and taking $r \rightarrow 0$,
$$\|g\|_{L^n(B_{1/4}(x_0))} = \|\partial_\nu f\|_{L^n(B_{r/4}(rx_0))} \rightarrow 0.$$
On the other hand, by the Poincaré inequality,
$$\|f(rx) - f(rx_0)\|_{L^n(B_{1/4}(x_0))} = \frac{\|f - f(x_0)\|_{L^n(B_{r/4}(rx_0))}}{r} \leq C\|\nabla f\|_{L^n(B_{r/4}(rx_0))} \rightarrow 0.$$

Hence, choosing $r$ small enough, ve can have $v(x_0) \geq \min\{M/2,\eta/(64n)\}/2$, which contradicts $v(x_0) < 0$.
\end{proof}

Using the lemma, we prove that there is an arbitrarily wide cone of directions where $\partial_\nu u_r \geq 0$, for small $r > 0$.

\begin{proof}[Proof of Proposition \ref{prop:classic_FB_Lipschitz}]
We only need to check that, for any $\nu \in \mathbb{S}^{n-1}$ such that ${\nu \cdot e > c(L)}$, the hypotheses of the lemma hold. By construction, we only need to check that (\ref{eq:condicions_lema_positivitat}) holds for a small enough $r > 0$.

Let $\delta = \varepsilon^{1/8}$. By the blow-up, there exists $r > 0$ such that
$$\left|u_r - \frac{\gamma}{2}(x\cdot e)_+^2\right| < \varepsilon,$$
and
$$\big|\partial_\nu u_r - \gamma(x\cdot e)_+(x\cdot \nu)\big| < \varepsilon.$$

Hence, if $\varepsilon > 0$ is small enough,
$$u_r > \frac{\gamma}{2}(x\cdot e)_+^2 - \varepsilon > \frac{\gamma}{2}\delta^4 - \varepsilon = \frac{\gamma}{2}\varepsilon^{1/2} - \varepsilon > 0 \quad \text{in} \quad \{x \cdot e > \delta^2\}.$$
Moreover,
$$u_r = 0 \quad \text{in} \quad \{x \cdot e < -\delta^2\},$$
as we prove by contradiction from the nondegeneracy. Suppose $u_r(y) > 0$ for some $y$ such that $y \cdot e < -\delta$. Then,
$$\sup\limits_{B_{\delta^2}(y)}u_r \geq c\delta^4 = c\varepsilon^{1/2},$$
but since $B_{\delta^2}(y) \subset \{x \cdot e < 0\}$, $u_r < \varepsilon$, which cannot happen if $\varepsilon$ is small enough. Hence, the free boundary is contained in the strip $\{|x \cdot e| < \delta^2\}$.

Now, let a unit $\nu$ such that $\nu\cdot e > c(L)$. The lower bound for $\partial_\nu u_r$ in $N_\delta$ only takes into account the convergence of the blow-up,
$$\partial_\nu u_r > \gamma c(L)(x\cdot e)_+ - \varepsilon \geq -\varepsilon.$$
On the other hand, if $z \in \Omega \setminus N_\delta$, since the free boundary is at a distance lower than $\delta^2$ from the hyperplane $\{x\cdot e = 0\}$, $z \cdot e > \delta - \delta^2$. Hence,
$$\partial_\nu u_r(z) > \gamma c(L)(z \cdot e)_+ - \varepsilon > \gamma c(L)(\delta - \delta^2) - \varepsilon = \gamma c(L)(\varepsilon^{1/8} - \varepsilon^{1/4}) - \varepsilon =: M,$$
where $M > 0$ provided that $\varepsilon$ is small enough.

Notice that $r$ and $\varepsilon$ (thus $\delta$) are uniform in $\nu$. Now, applying the previous Lemma \ref{lem:classical_obstacle_derivada_positiva}, for all unit $\nu$ such that $\nu \cdot e > c(L)$, $\partial_\nu u_r \geq 0$, which is equivalent to $\partial_\nu u \geq 0$ in $B_r$. Now, if $x_0 \in B_r$ is a free boundary point, $u(x_0) = 0$, hence $u(x_0 - t\nu) \leq 0$ whenever $x_0 - t\nu \in B_r$. Since $u \geq 0$, there is a cone \textit{behind} $x_0$ where $u = 0$, i.e.
$$u = 0 \quad \text{in} \quad B_r \cap \{(x - x_0) \cdot e < -c(L)|x|\}.$$
In the interior of the cone, there are no free boundary points because $u$ is $0$ in a neighbourhood of all points. This is the interior cone. To check the exterior cone condition, suppose there is another free boundary point $x_1$ in the set\linebreak $B_r \cap \{x_0 + t\nu : \nu \cdot e > c(L), t \in \R^+\}$. Then, by applying the interior cone condition to $x_1$, we get that $x_0$ cannot be a free boundary point, a contradiction. This proves that, in $B_r$, the free boundary is a Lipschitz graph with constant $L$ in the direction~$e$.
\end{proof}

Now we can use our new boundary Harnack inequality to prove the $\mathcal{C}^{1,\alpha}$ regularity of the free boundary at regular points \textit{à la} Caffarelli. To do this, we must ask the right hand side to belong to $W^{1,q}$ with $q > n$, which is slightly more restrictive and implies that $f$ is Hölder continuous.

\begin{proof}[Proof of Corollary \ref{cor:classical_obstacle_lq}]
As it is customary in this kind of proof, we will use the boundary Harnack with the derivatives of $u_r$. Let $L = L_0(q,n,1,1)/2$ with the $L_0$ defined in Corollary \ref{cor:BH_rhs_symmetric}. From Proposition \ref{prop:classic_FB_Lipschitz}, there exists $r > 0$ such that the free boundary is a Lipschitz graph with constant $L$ in $B_r$. Assume without loss of generality that the direction of the graph is $e = e_n$, and that $L < 1$.

For $i = 1,\ldots,n-1$, consider the functions
$$w_1 = \partial_i u_r \quad \text{and} \quad w_2 = \partial_n u_r.$$

They are both solutions of $\Delta w_j = g_j$, with $g_1(x) = r\partial_i f(rx)$, $g_2(x) = r\partial_n f(rx)$. Moreover, $w_2$ is positive. To be able to use the boundary Harnack, we need to see that the right hand is small. Indeed, taking $r \rightarrow 0$,
$$\|g_j\|_{L^q(B_1)} \leq \|r\nabla f(rx)\|_{L^q(B_1)} = 2r^{1-n/q}\|\nabla f\|_{L^q(B_r)} \rightarrow 0.$$
Finally, by the blow-up convergence, 
$$w_j(e_n/2) > \gamma/2 - \varepsilon > \gamma/4, \quad w_j(e_n/2) < \gamma/2 + \varepsilon < \gamma.$$
Thus, we can normalize $w_j$ dividing by $w_j(e_n/2)$ and the right hand side still converges to $0$ in norm.

Let $\Omega_r = \{u_r > 0\}$. By the boundary Harnack with right hand side, Theorem \ref{thm:BH_rhs},
$$\frac{w_1}{w_2} \in \mathcal{C}^{0,\alpha}(B_{1/2}\cap\Omega_r) \quad \Rightarrow \quad \frac{\partial_i u_r}{\partial_n u_r} \in \mathcal{C}^{0,\alpha}(B_{1/2}\cap\Omega_r).$$

The unit normal vector to any level set $\{u_r = t\}$, $t > 0$, is, by components,
$$\hat{n}^i = \frac{\partial_i  u_r}{|\nabla u_r|} = \frac{\partial_i u_r / \partial_nu_r}{\sqrt{1 + \sum_{j = 1}^{n - 1} (\partial_j u_r / \partial_nu_r)^2}} \in \mathcal{C}^{0,\alpha}(B_{1/2}\cap\Omega_r).$$

As this expression is $\mathcal{C}^{0,\alpha}$ up to the boundary, this proves the normal vector to the free boundary is $\mathcal{C}^{0,\alpha}$, and by a simple calculation it follows that the free boundary is $\mathcal{C}^{1, \alpha}$.
\end{proof}

\subsection{$\mathcal{C}^{1,\alpha}$ regularity of the free boundary in the fully nonlinear obstacle problem}
Consider the fully nonlinear obstacle problem in the general version (\ref{eq:fully_nonlinear}), under the assumptions in Corollary \ref{cor:fully_nonlinear_thick}.

Our starting point will be the existence of a regular blow-up in the sense of Definition \ref{defn:regular_point_classical}, i.e., there exists $r_k \downarrow 0$ such that
$$\frac{u(r_kx)}{r_k^2} \rightarrow \frac{\gamma}{2}(x\cdot e)_+^2 \quad \text{in } \mathcal{C}^1_{\mathrm{loc}}(\R^n)$$
for some $\gamma > 0$ and $e \in \mathbb{S}^{n-1}$. We also need the classical nondegeneracy condition: if $x_0 \in \overline{\{u > 0\}}$,
$$\sup\limits_{B_r(x_0)} u \geq cr^2.$$

From here, we will extend the proof of \cite{IM16} to the case where $f \in W^{1,q}$ (and not necessarily Lipschitz), and we will also prove $\mathcal{C}^{1,\alpha}$ regularity of the free boundary.

Our first step is an analogue to \cite[Lemma 3.7]{IM16} for the case $f \in W^{1,n}$.

\begin{lem}\label{lem:indrei_minne_modificat}
Let $u$ be a solution of
\begin{equation*}
\left\{
\begin{array}{rcll}
F(D^2u(x),rx) & = & f(rx)\chi_{\{u > 0\}} & \text{a.e.\ in } B_1\\
u & \geq & 0 & \text{a.e.\ in } B_1.
\end{array}
\right.
\end{equation*}
Assume that the conditions (H1), (H2) and (H3) from Corollary \ref{cor:fully_nonlinear_thick} hold. If
$$C_0\p_\nu u - u \geq - \varepsilon \quad \text{in} \quad B_1,$$
for any $C_0 > 0$, then
$$C_0\p_\nu u - u \geq 0 \quad \text{in} \quad B_{1/2},$$
provided that $r, \varepsilon > 0$ are sufficiently small.
\end{lem}

The proof is the same as in our source, except for the final step. We provide it here for the convenience of the reader.

\begin{obs}
For this lemma, the case $q = n$, i.e., when $f \in W^{1,n}$ and $F \in W^{1,n}$ with respect to the second variable, is also true.
\end{obs}

\begin{proof}
Let $x \in \{u > 0\}$ and $\p_1 F(M, x)$ denote the subdifferential of $F$ at the point $(M,x)$ with respect to the first variable. Since $F$ is convex in $M$, then ${\p_1 F(M, x) \neq \varnothing}$. Consider a measurable function $P^M$ with $P^M(x) \in \p_1 F(M,x)$. Since $f \in \mathcal{C}^\alpha$, by interior regularity estimates $u \in \mathcal{C}^{2,\alpha}_{\mathrm{loc}}(\{u > 0\})$, and thus we can define the measurable coefficients
$$a_{ij}(x) := (P^{D^2u(x)}(rx))_{ij} \in \p_1 F(D^2u(x),rx).$$

Since $F$ is convex in the first variable and $F(0,x) \equiv 0$, then for any unit vector $\nu$,
\begin{align*}
&\sum_{i,j=1}^na_{ij}(x)\frac{\p_{ij}u(x+h\nu) - \p_{ij}u(x)}{h} \leq \frac{F(D^2u(x+h\nu),rx) - F(D^2u(x),rx)}{h},\\
&\sum_{i,j=1}^na_{ij}(x)\p_{ij}u(x) = F(0,rx) + a_{ij}\p_{ij}u(x) \geq F(D^2u(x),rx) = f(rx),
\end{align*}
provided that $x+h\nu \in \{u > 0\} \cap B_1$. Now, since uniform limits of $L^n$-viscosity solutions are $L^n$-viscosity solutions (\cite[Theorem 3.8]{CCKS96}),
\begin{align*}
\sum_{i,j=1}^na_{ij}(x)\p_{ij}\p_\nu u(x) &\leq \limsup\limits_{h\rightarrow 0}\frac{F(D^2u(x+h\nu),rx) - F(D^2u(x),rx)}{h}\\
&= \limsup\limits_{h\rightarrow 0}\frac{F(D^2u(x+h\nu),rx) - f(rx)}{h}\\
&= \limsup\limits_{h\rightarrow 0}\frac{F(D^2u(x+h\nu),rx) - F(D^2u(x+h\nu),rx+rh\nu)}{h}+\\
&\phantom{= \limsup}+\frac{f(rx+rh\nu)-f(rx)}{h}\\
&= r(\p_\nu f)(rx) - r(\p_{2,\nu}F)(D^2u(x),rx)
\end{align*}
in the $L^n$-viscosity sense.

Suppose there exists $y_0 \in \Omega \cap B_{1/2}$ such that $C_0\partial_\nu u(y_0) - u(y_0) < 0$. Then, we consider the auxiliary function
$$w(x) = C_0\partial_\nu u(x) - u(x) + \tau_0\frac{|x-y_0|^2}{4n\Lambda}.$$
Then,
\begin{equation*}
\begin{split}
a_{ij}\p_{ij}w(x) &\leq rC_0(\partial_\nu f)(rx) - rC_0(\partial_{2,\nu}F)(D^2u(x),rx) - f(rx) + \tau_0/2\\
&\leq rC_0(\partial_\nu f(rx) - \partial_{2,\nu}F(D^2u(x),rx)) - f(rx) + \tau_0/2\\
&\leq rC_0(\partial_\nu f(rx) - \partial_{2,\nu}F(D^2u(x),rx)) =: rR(rx).
\end{split}
\end{equation*}
Hence, by the ABP estimate, since $R \in L^n(B_1)$,
$$0 > \inf\limits_{\Omega \cap B_{1/4}(y_0)} w \geq \inf\limits_{\p(\Omega \cap B_{1/4}(y_0))} w - C\|rR(rx)\|_{L^n(\Omega \cap B_{1/4}(y_0))}.$$
By the scaling of the $L^n$ norm, the second term in the sum is bounded by
$$C\|R\|_{L^n(B_{r/4}(ry_0))} \rightarrow 0$$
as $r \rightarrow 0$. On the other hand, $w \equiv 0$ on $\p\Omega$, and
$$w \geq -\varepsilon + \frac{\tau_0}{64n\Lambda} \quad \text{on} \quad \Omega \cap \partial B_{1/4}.$$

Therefore, choosing $\varepsilon$ and $r$ small enough we reach $w > 0$ in $\Omega \cap B_{1/4}(y_0)$, a contradiction.
\end{proof}

Now, as we show next, by the $\mathcal{C}^1$ convergence of the blow-up we can fulfill the sufficient conditions in Lemma \ref{lem:indrei_minne_modificat}, and prove that the free boundary is Lipschitz at regular points, analogously to Proposition \ref{prop:classic_FB_Lipschitz}. Then, applying the boundary Harnack inequality, we can improve the regularity up to $\mathcal{C}^{1,\alpha}$.

We denote $u_r(x) := r^{-2}u(rx)$ as in Proposition \ref{prop:classic_FB_Lipschitz}.
\begin{proof}[Proof of Corollary \ref{cor:fully_nonlinear_thick}]
Let
$$u_0(x) = \frac{\gamma}{2}(x\cdot e)^2_+$$
be the blow-up at 0. Let $s \in (0,1)$. Then,
$$\frac{\partial_\nu u_0}{s} - u_0 = \gamma\left(\frac{(x\cdot e)_+(e \cdot \nu)}{s} - \frac{(x\cdot e)_+^2}{2}\right) \geq 0$$
for any direction $\nu \in \mathbb{S}^{n-1}$ such that $\nu \cdot e \geq s/2$. From the $\mathcal{C}^1$ convergence of the blow-up, there exists $r_k$ such that
$$\frac{\partial_\nu u_{r_k}}{s} - u_{r_k} \geq -\varepsilon \quad \text{in} \quad B_1.$$

By Lemma \ref{lem:indrei_minne_modificat}, this implies
$$\frac{\partial_\nu u_{2\rho}}{s} - u_{2\rho} \geq 0 \quad \text{in} \quad B_{1/2},$$
for some sufficiently small $\rho > 0$.

In particular, this shows that the free boundary fulfills the interior and exterior cone conditions in $B_\rho$ and therefore it is Lipschitz, with Lipschitz constant $L(s)$, that satisfies $L(s) \rightarrow 0$ as $s \rightarrow 0$.

Now, assume without loss of generality $e = e_n$. For $i = 1,\ldots,n-1$, consider the functions
$$w_1 = \partial_i u_\rho \quad \text{and} \quad w_2 = \partial_n u_\rho.$$

Notice that $w_2 \geq 0$. Since $F$ is $\mathcal{C}^1$ with respect to $D^2u$ and $F(D^2u,x) \in W^{1,q}$, then $u \in W^{3,q}$ and we can commute the third derivatives as follows,
\begin{align*}
\partial_\nu F(D^2u_\rho(x), \rho x) &= \sum\limits_{i,j=1}^nF_{ij}\partial_\nu\partial_{ij}u_\rho + \rho\p_{2,\nu}F = \sum\limits_{i,j=1}^nF_{ij}\partial_{ij}(\partial_\nu u_\rho) + \rho\p_{2,\nu}F = \rho\partial_\nu f,\\
\mathcal{L}(\p_\nu u_\rho) &= \rho(\p_\nu f - \p_{2,\nu}F).
\end{align*}

Here, $\mathcal{L}w = \operatorname{Tr}(A(x)w)$, with $A(x) = (F_{ij}(D^2 u_\rho, \rho x))_{ij}$. Hence, $w_1$ and $w_2$ are both solutions of
$$\mathcal{L}w_1 = g_1 := \rho(\partial_i f - \p_{2,i}F)(\rho x) \quad \text{and} \quad \mathcal{L}w_2 = g_2 := \rho(\p_n f - \p_{2,n}F)(\rho x).$$

To be able to use the boundary Harnack, we need to show that the right hand is small. Indeed, taking $\rho \rightarrow 0$,
$$\|g_j\|_{L^q(B_1)} \leq \|\rho(\nabla f(\rho x) + \nabla_2 F(\rho x))\|_{L^q(B_1)} = \rho^{1-n/q}\|\nabla f + \nabla_2 F\|_{L^q(B_\rho)} \rightarrow 0.$$

Finally, by the blow-up convergence, 
$$w_j(e_n/2) > \gamma/2 - \varepsilon > \gamma/4, \quad w_j(e_n/2) < \gamma/2 + \varepsilon < \gamma.$$
Thus, we can normalize $w_j$ dividing by $w_j(e_n/2)$ and the right hand side still converges to $0$ in norm.

Let $\Omega_\rho = \{u_\rho > 0\}$. By the boundary Harnack with right hand side, Theorem \ref{thm:BH_rhs},
$$\frac{w_1}{w_2} \in \mathcal{C}^{0,\alpha}(\Omega_\rho\cap B_{1/2}) \quad \Rightarrow \quad \frac{\partial_i u_\rho}{\partial_n u_\rho} \in \mathcal{C}^{0,\alpha}(\Omega_\rho\cap B_{1/2}).$$

The unit normal vector to any level set $\{u_\rho = t\}$, $t > 0$, is, by components,
$$\hat{n}^i = \frac{\partial_i  u_\rho}{|\nabla u_\rho|} = \frac{\partial_i u_\rho / \partial_nu_\rho}{\sqrt{1 + \sum_{j = 1}^{n - 1} (\partial_j u_\rho / \partial_nu_\rho)^2}} \in \mathcal{C}^{0,\alpha}(\Omega_\rho\cap B_{1/2}).$$

As this expression is $\mathcal{C}^{0,\alpha}$ up to the boundary, this proves the normal vector to the free boundary is $\mathcal{C}^{0,\alpha}$, and it follows that the free boundary is $\mathcal{C}^{1, \alpha}$.
\end{proof}

\subsection{$\mathcal{C}^{1,\alpha}$ regularity of the free boundary in the fully nonlinear thin obstacle problem}
Recall the fully nonlinear thin obstacle problem (\ref{eq:fully_nonlinear_signorini}), under the assumptions in Corollary \ref{cor:fully_nonlinear_thin}.

We will denote $u = v - \varphi$. In this case, we know the following.

\begin{prop}[\cite{RS16}]
Assume that $0$ is a regular free boundary point for (\ref{eq:fully_nonlinear_signorini}), where $F$ is uniformly elliptic, convex and $F(0) = 0$, and $\varphi \in \mathcal{C}^{1,1}$. Then, there exists $e \in \mathbb{S}^{n-1} \cap \{x_n = 0\}$ such that for any $L > 0$ there exists $r > 0$ for which
$$\partial_\nu u \geq 0 \quad \text{in} \quad B_r \quad \text{for all} \quad \nu \cdot e \geq \frac{L}{\sqrt{L^2+1}}, \quad \nu \in \mathbb{S}^{n-1}\cap\{x_n = 0\}.$$
In particular, the free boundary is Lipschitz in $B_r$, with Lipschitz constant $L$.
\end{prop}

Now, using our new boundary Harnack in slit domains, Theorem \ref{thm:BH_rhs_slit}, on top of this proposition, we derive the $\mathcal{C}^{1,\alpha}$ regularity of the free boundary at regular points.

\begin{proof}[Proof of Corollary \ref{cor:fully_nonlinear_thin}]
Let $\Omega = B_1 \setminus \{(x',0) : u(x',0) = 0\}$. The free boundary is a Lipschitz graph inside $B_r \cap \{x_n = 0\}$. Suppose without loss of generality that the direction of the graph is $e = e_{n - 1}$. Choosing $L$ and $r$ small enough, for all $\nu \in \mathbb{S}^{n} \cap \{x_n = 0\}$ such that $\nu\cdot e_{n-1} \geq 1/2$, $\partial_\nu u \geq 0$ in $B_r$.

For $i = 1, \ldots, n - 2$, consider the functions
$$w_1 = \p_iu, \quad w_2 = \p_{n-1}u.$$

Since $F \in \mathcal{C}^1$ and $F(D^2v) = 0$, then $v \in W^{3,p}$ for all $p < \infty$ and we can commute the third derivatives as follows,
\begin{equation*}
\begin{split}
   \partial_\nu(F(D^2v)) &= 0 \quad \text{in} \quad \Omega,\\
   \partial_\nu(F(D^2v)) &= \sum\limits_{i, j = 1}^n F_{ij}\partial_\nu(\partial_{ij}v) = \sum\limits_{i, j = 1}^n F_{ij}\partial_{ij}(\partial_\nu v) = \operatorname{Tr}(A D^2(\partial_\nu v)),
\end{split}
\end{equation*}

Moreover, $w_2$ is positive. Then, using that $v = u + \varphi$,
\[
\left\{
\begin{array}{rcll}
\mathcal{L}w_1 & = & -\mathcal{L}(\partial_i\varphi) & \text{in } \Omega\\
w_1 & = & 0 & \text{on } B_1 \setminus \Omega,
\end{array}
\right.
\quad and\quad
\left\{
\begin{array}{rcll}
\mathcal{L}w_2 & = & -\mathcal{L}(\p_{n-1}\varphi) & \text{in } \Omega\\
w_2 & = & 0 & \text{on } B_1 \setminus \Omega.
\end{array}
\right.
\]
where $\mathcal{L}w = \operatorname{Tr}(AD^2w)$, $A = (F_{ij} \circ D^2u)_{ij}$. Now, we will check that, after a scaling, $w_2(e_n/2) \geq 1$ and the right hand side becomes arbitrarily small. Define
$$\tilde{w}_2(x) = \frac{w_2(s x)}{s} \quad \text{and} \quad \tilde{\varphi}(x) = \frac{\varphi(s x)}{s^2}.$$

Now, we check that the right hand side is as small as required. Indeed, letting $s_k \rightarrow 0$,
\begin{align*}
\|\mathcal{L}(\partial_{n-1}\tilde{\varphi})\|_{L^q(B_1)} &\leq \Lambda\|D^3\tilde{\varphi}\|_{L^q(B_1)} = \Lambda\|s_k\varphi(s_kx)\|_{W^{3,q}(B_1)}\\
&= \Lambda s_k^{1-n/q}\|\varphi\|_{W^{3,q}(B_{s_k})} \rightarrow 0.
\end{align*}

The right hand side becomes arbitrarily small in the equation for $w_1$ analogously. Then, since $0$ is a regular free boundary point, by the convergence of the blow-up, 
$$\tilde{w}_2(e_{n-1}/2) \rightarrow \infty$$
for a sequence of values $\{s_k\} \rightarrow 0$. Now, by the interior Harnack inequality combined with the ABP estimate, since $\tilde{w}_2 \geq 0$ in $\Omega$ and the distance between the segment joining $e_{n-1}/2$ and $\pm e_n/2$ and the contact set is positive and larger than some constant $c(L,n)$ only depending on the Lipschitz constant of the free boundary and the dimension,
$$\tilde{w}_2(\pm e_n/2) \geq c_1\tilde{w}_2(e_{n-1}/2) - c_2\|\mathcal{L}(\partial_{n-1}\tilde{\varphi})\|_{L^n(B_1)} \geq c_1\tilde{w}_2(e_{n-1}/2) - c_2\Lambda\|\varphi\|_{W^{3,n}},$$
for some positive $c_1$ and $c_2$ only depending on the dimension, $L$, $\lambda$ and $\Lambda$.

Therefore, letting $s_k \rightarrow 0$, $\tilde{w}_2(\pm e_n/2) \geq 1$. If $\|\tilde{w}_1\|_{L^1} > 1$, we normalize it (notice that this step can only make the right hand side smaller). Thus, by the boundary Harnack inequality with right hand side for slit domains, Theorem \ref{thm:BH_rhs_slit}, $w_1/w_2 \in \mathcal{C}^{0, \alpha}$ in $\Omega \cap B_{s/2}$. Thus, $\partial_i u / \partial_{n-1} u \in \mathcal{C}^{0, \alpha}$. 

Now, the unit normal vector to any level set in the \textit{thin space} $\{x_n = 0\} \cap \{u = t\}$ with $t > 0$ is, by components,
$$\hat{n}^i = \frac{\partial_i u}{\sqrt{\sum_{j = 1}^{n - 1}|\partial_j u|^2}} = \frac{\partial_i u / \partial_{n-1} u}{\sqrt{1 + \sum_{j = 1}^{n - 2} (\partial_j u / \partial_{n - 1}u)^2}} \in \mathcal{C}^{0,\alpha}$$

Then, letting $t \rightarrow 0^+$, we recover that the normal vector to the free boundary is $\mathcal{C}^{0,\alpha}$, and hence the free boundary is a $\mathcal{C}^{1,\alpha}$ graph.
\end{proof}

\section{Sharpness of the results}\label{sect:counterexample}

We construct two examples that show that:
\begin{itemize}
    \item Without the smallness assumption on the Lipschitz constant of the domain, Theorem \ref{thm:BH_rhs} fails.
    \item For $q = n$, Theorem \ref{thm:BH_rhs} fails.
    \item For divergence form operators, if the coefficients are only bounded and measurable, Theorem \ref{thm:BH_rhs} fails.
\end{itemize}

As a first observation, see \cite{AS19}, take for instance $\Omega = \{x_n > 0\} \subset \R^n$, and let $u(x) = x_n$, $v(x) = x_n^2$. These functions are normalized in the sense that\linebreak $u(e_n) = v(e_n) = 1$, and vanish contiuously on $\partial\Omega$. Even in a flat domain, a function with a too large Laplacian, $|\Delta v| = 2$, will never be comparable to a harmonic function near the boundary. Hence, the right hand side of the equation must be small, otherwise the result fails.

The following example in two dimensions shows that if we ask $\Delta v$ to be small in $L^q$ norm, for any $q > n$ there is a cone narrow enough such that we can find harmonic functions that are not comparable with $v$. Moreover, if we consider a fixed cone, there exists $q > n$ such that the $L^q$ boundedness of the right hand side is not enough to have a boundary Harnack. If $q = n$, such counterexamples are valid for any cone.

\begin{prop}\label{prop:lq_counterexample}
Let $L > 0$, $q > 0$, and assume
\begin{equation}\label{eq:contraexemple}
\frac{\pi}{2\arctan(1/L)} + \frac{2}{q} > 2.
\end{equation}

Then, for every $\delta > 0$, there exists a cone $\Omega \subset \R^2$ with Lipschitz constant $L$, and positive functions $u, v$ that vanish continuously on $\partial\Omega$ such that
$$u(0,1) = v(0,1) = 1, \quad \Delta u = 0\quad \text{and} \quad \|\Delta v\|_{L^q(\Omega)} < \delta,$$
but
$$\sup\limits_\Omega \frac{u}{v} = \infty.$$

In particular, Theorem \ref{thm:BH_rhs} fails for $q = n$.
\end{prop}

\begin{proof}
Consider the cone $\Omega = \{(x,y) \in \R^2 : y > L|x|\}$ and let 
$$\beta = \frac{\pi}{2\arctan(1/L)} \quad \text{and} \quad u(x, y) = \operatorname{Re}\left((-ix+y)^\beta\right).$$

Then, $u$ is harmonic, positive in $\Omega$, and vanishes continuously on $\partial\Omega$. Let $\psi$ be a positive smooth function such that
$$0 \leq \psi \leq u, \quad \operatorname{Supp} \psi \subset B_{1/3}(0,1) \quad \text{and} \quad \psi(0,1) = u(0,1) = 1.$$

Define the scalings $\psi_\varepsilon(x,y) = \varepsilon^\beta\psi(x/\varepsilon,y/\varepsilon)$. Since $u$ is homogeneous of degree $\beta$, $0 \leq \psi_\varepsilon \leq u$. Moreover,
$$\Delta \psi_\varepsilon(x, y) = \varepsilon^{\beta - 2}(\Delta\psi)(x / \varepsilon, y / \varepsilon).$$

Now, we construct $v$ as the following infinite sum, that converges uniformly.
$$v := u - \sum\limits_{k = k_0}^\infty(1-2^{-k})\psi_{2^{-k}}.$$
Since the supports of $\psi_{2^{-k}}$ are disjoint, $v \geq 0$. On the other hand,
\begin{align*}
&\left\|\Delta v\right\|_{L^q(\Omega)} \leq \sum\limits_{k = k_0}^\infty\|\Delta\psi_{2^{-k}}\|_{L^q(\Omega)} = \sum\limits_{k = k_0}^\infty2^{-k(\beta-2+2/q)}\|\Delta\psi\|_{L^q(\Omega)} \rightarrow 0
\end{align*}
as $k_0 \rightarrow \infty$, since $\beta + 2/q > 2$ by hypothesis. Hence, we can choose $k_0$ big enough so that $\|\Delta v\|_{L^q(\Omega)} < \delta$.

To end, for $k \geq k_0$,
$$\frac{u(2^{-k},0)}{v(2^{-k},0)} = 2^k \rightarrow \infty,$$
as wanted.
\end{proof}

\begin{obs}
Since $\arctan(1/L) \in (0,\pi/2)$, the first term in the condition (\ref{eq:contraexemple}) is always greater than $1$, hence, if $q \leq 2$ there are always counterexamples to the boundary Harnack with right hand side bounded in $L^q$.

On the other hand, if $L > 1$, $\arctan(1/L) < \pi/4$, and the condition is fulfilled for all $q > 0$ and $q = \infty$.

The limiting case $L = 0$, $q = 2$ (or $q = n$ in higher dimensions) corresponds to domains that are locally a half-space. We have not considered this particular case in our setting.
\end{obs}

The existence of such example shows that, to have a boundary Harnack inequality for equations with a right hand side, we need the Lipschitz constant of the boundary to be sufficiently small, and also the right hand side to be small compared to the values of the function. It also shows a trade-off between the maximum possible slope of the boundary and the exponent of the $L^q$ boundedness of the right hand side.

On the other hand, it is impossible to have a boundary Harnack for equations with right hand side in Lipschitz domains with \textit{narrow} corners, even less in Hölder domains or more general domains, under the reasonable hypothesis ${\Delta u = f \in L^\infty}$, with $\|f\|_{L^\infty}$ small. 

However, the boundary Harnack holds for divergence form operators in Lipschitz domains with big Lipschitz constants when the right hand side vanishes as a big enough power of the distance \cite{AS19}. It is likely that we could prove the same for non-divergence form operators, but we will not pursue this because it cannot be used in the context of free boundary problems.

The following example is based in a counterexample to the Hopf lemma for divergence operators with discontinuous coefficients \cite{Naz12}, and shows that the boundary Harnack for equations with a right hand side fails in this setting.

\begin{prop}\label{prop:counterexample_div_form}
There exists $\mathcal{L}$ in divergence form with discontinuous coefficients and positive functions $u, v$ in $\{y > 0\} \subset \R^2$ that vanish continuously at $\{y = 0\}$ such that
$$u(1,1) = v(1,1) = 1, \quad \mathcal{L}u = 0 \text{ in } \{y > 0\} \quad \text{and} \quad \|\mathcal{L}v\|_{L^\infty(B_1^+)} < \delta,$$
for any given $\delta > 0$, but
$$\sup\limits_{B_{1/2}^+}\frac{u}{v} = \infty.$$
In particular, Theorem \ref{thm:BH_rhs} fails if the divergence form operator has discontinuous coefficients.
\end{prop}

\begin{proof}
Let $\mathcal{L}u = \operatorname{Div}(A(x,y)\nabla u)$, with
$$A(x,y) = \begin{pmatrix}
1 & -6\operatorname{sgn}(x)\\
-6\operatorname{sgn}(x) & 48
\end{pmatrix}.$$
It is easy to check that $\mathcal{L}$ is uniformly elliptic and that 
$$u(x,y) = \frac{y^3 + 18|x|y^2 + 72x^2y}{91}$$
is a solution of $\mathcal{L}u = 0$. Now we will define $v$ as a perturbation of $u$, in a similar way as in Proposition \ref{prop:lq_counterexample}. We will use that the coefficients $A(x,y)$ are constant in the positive quadrant.

Let $\psi$ be a positive smooth function such that
$$0 \leq \psi \leq u, \quad \operatorname{Supp} \psi \subset B_{1/3}(1,1) \quad \text{and} \quad \psi(1,1) = u(1,1) = 1.$$

Define the scalings $\psi_\varepsilon(x,y) = \varepsilon^3\psi(x/\varepsilon,y/\varepsilon)$. Since $u$ is homogeneous of degree $3$, $0 \leq \psi_\varepsilon \leq u$. Moreover,
$$\mathcal{L} \psi_\varepsilon(x, y) = \varepsilon(\mathcal{L}\psi)\left(\frac{x}{\varepsilon},\frac{y}{\varepsilon}\right).$$

Now, we construct $v$ as the following infinite sum, that converges uniformly.
$$v := u - \sum\limits_{k = k_0}^\infty(1-2^{-k})\psi_{2^{-k}}.$$
Since the supports of $\psi_{2^{-k}}$ are disjoint, $v \geq 0$. On the other hand,
\begin{align*}
&\left\|\mathcal{L} v\right\|_{L^\infty(B_1^+)} \leq \sum\limits_{k = k_0}^\infty\|\mathcal{L}\psi_{2^{-k}}\|_{L^\infty(B_1^+)} = \sum\limits_{k = k_0}^\infty2^{-k}\|\mathcal{L}\psi\|_{L^\infty(B_1^+)} \rightarrow 0
\end{align*}
as $k_0 \rightarrow \infty$. Hence, we can choose $k_0$ big enough so that $\|\mathcal{L} v\|_{L^\infty(B_1^+)} < \delta$.

To end, for $k \geq k_0$,
$$\frac{u(2^{-k},2^{-k})}{v(2^{-k},2^{-k})} = 2^k \rightarrow \infty,$$
as wanted.
\end{proof}

\section{Hopf lemma for non-divergence equations with right hand side}\label{sect:hopf}
We now recall the classical Hopf lemma in a very general version for non-divergence elliptic equations \cite{LZ18}.

\begin{thm}\label{thm:hopf}
Suppose that $\Omega$ satisfies the interior $\mathcal{C}^{1,\mathrm{Dini}}$ condition at $0 \in \partial\Omega$ and $u \in \mathcal{C}(\overline{\Omega}\cap B_1)$ satisfies
$$\mathcal{M}^-(D^2u) \leq 0 \quad \text{in} \quad \Omega \cap B_1$$
in the $L^n$-viscosity sense with $u(0) = 0$ and $u \geq 0$ in $\Omega\cap B_1$.

Then for any $l = (l_1,\ldots,l_n) \in \R^n$ with $|l| = 1$ and $l_n > 0$,
$$u(rl) \geq cl_nu(e_n/2)r, \quad r \in (0,\delta),$$
where $c > 0$ and $\delta$ depend only on the dimension, $\lambda$, $\Lambda$ and the modulus of continuity of the domain.
\end{thm}

We can use this result to prove a generalized Hopf lemma for the solutions of non-divergence equations with small right hand side.

\begin{cor}
Let $q > n$ and $\mathcal{L}$ in non-divergence form as in (\ref{eq:non-divergence_operator}). There exist small $c_0 > 0$ and $L_0 > 0$ such that the following holds.

Let $\Omega$ be a Lipschitz domain as in Definition \ref{defn:lipschitz_domain}, with Lipschitz constant $L < L_0$. Suppose further that $\p\Omega$ is a $\mathcal{C}^{1,\mathrm{Dini}}$ graph. Let $v$ be a solution of
$$\left\{
\begin{array}{rcll}
\mathcal{L}v & = & f & \text{in } \Omega \cap B_1\\
v & = & 0 & \text{on } \partial\Omega \cap B_1
\end{array}
\right.$$
in the $L^n$-viscosity, with $v > 0$ in $\Omega\cap B_1$ and
$$\|f\|_{L^q(B_1)} \leq c_0v(e_n/2).$$

Then, for any $l = (l_1,\ldots,l_n) \in \R^n$ with $|l| = 1$ and $l_n > 0$,
$$v(rl) \geq cl_nv(e_n/2)r, \quad r \in (0,\delta),$$
where $c$, $c_0$ and $\delta$ are positive and depend only on the dimension, $\lambda$, $\Lambda$ and the modulus of continuity of the domain.
\end{cor}

\begin{proof}
Assume $v(e_n/2) = 1$ without loss of generality. Let $u$ be a positive solution of the Dirichlet problem
\begin{equation*}
\left\{
\begin{array}{rcll}
\mathcal{L}u & = & 0 & \text{in } \Omega \cap B_1\\
u & = & 0 & \text{on } \partial\Omega \cap B_1.
\end{array}
\right.
\end{equation*}
After dividing by a constant, $u(e_n/2) \leq 1$.

Now, by Theorem \ref{thm:BH_rhs}, we have $u \leq Cv$ in $B_{1/2}$, hence the estimate of Theorem \ref{thm:hopf} for $u$ is also valid for $Cv$, and the result follows.
\end{proof}

\end{document}